\pdfoutput=1
\RequirePackage{ifpdf}
\ifpdf 
\documentclass[pdftex]{sigma}
\else
\documentclass{sigma}
\fi

\newtheorem{Theorem}{Theorem}[section]

\newtheorem{Conjecture}[Theorem]{Conjecture}
{ \theoremstyle{definition}
\newtheorem{Definition}[Theorem]{Definition}

\newtheorem{Remark}[Theorem]{Remark}
\newtheorem{Remarks}[Theorem]{Remarks}}

\usepackage{amstext,verbatim,mathrsfs}
\usepackage[mathcal]{eucal}
\usepackage[all]{xy}

\makeatletter
\def\makeCal#1{\expandafter\newcommand\csname c#1\endcsname{\mathcal{#1}}}
\def\makeBB#1{\expandafter\newcommand\csname b#1\endcsname{\mathbb{#1}}}
\def\makeFrak#1{\expandafter\newcommand\csname f#1\endcsname{\mathfrak{#1}}}
\count@=0 \loop \advance\count@ 1 \edef\y{\@Alph\count@}
\expandafter\makeCal\y \expandafter\makeBB\y \expandafter\makeFrak\y
\ifnum\count@<26 \repeat

\newcommand{\ignore}[1]{}

\newcommand{\im}{\operatorname{im}}

\newcommand{\Stab}{\operatorname{Stab}}

\renewcommand{\Re}{\operatorname{Re}}

\newcommand{\End}{\operatorname{End}}

\newcommand {\<}{\langle}
\renewcommand {\>}{\rangle}
\newcommand{\half}{\frac{1}{2}}
\newcommand{\tensor}{\otimes}
\renewcommand{\O}{\mathscr{O}}
\newcommand{\isom}{\cong}
\newcommand{\into}{\hookrightarrow}

\newcommand{\hk}{hyperk{\"a}hler }

\newcommand{\lra}{\longrightarrow}
\newcommand{\lRa}[1]{\stackrel{#1}{\longrightarrow}}
\newcommand{\lLa}[1]{\stackrel{#1}{\longleftarrow}}
\newcommand{\Lie}{\cL}

\newcommand{\PGL}{\operatorname{PGL}}
\newcommand{\MCG}{\operatorname{MCG}}
\newcommand{\hash}{\#}
\renewcommand{\leq}{\leqslant}
\newcommand{\res}{\operatorname{res}}

\newcommand{\TB}{\scriptstyle{TB}}
\newcommand{\LR}{\scriptstyle{LR}}
\newcommand{\GL}{\operatorname{GL}}
\newcommand{\Quad}{\operatorname{Quad}}

\numberwithin{equation}{section}

\begin{document}

\allowdisplaybreaks

\newcommand{\arXivNumber}{2303.07061}

\renewcommand{\PaperNumber}{112}

\FirstPageHeading

\ShortArticleName{Tau Functions from Joyce Structures}

\ArticleName{Tau Functions from Joyce Structures}

\Author{Tom BRIDGELAND}

\AuthorNameForHeading{T.~Bridgeland}

\Address{Department of Pure Mathematics, University of Sheffield, Sheffield, S3\ 7RH, UK}
\Email{\href{mailto:t.bridgeland@sheffield.ac.uk}{t.bridgeland@sheffield.ac.uk}}

\ArticleDates{Received July 26, 2024, in final form December 12, 2024; Published online December 18, 2024}

\Abstract{We argued in [\textit{Proc. Sympos. Pure Math.}, Vol.~103, American Mathematical Society, Providence, RI, 2021, 1--66, arXiv:1912.06504] that, when a certain sub-exponential growth property holds, the Donaldson--Thomas invariants of a 3-Calabi--Yau triangulated category should give rise to a geometric structure on its space of stability conditions called a Joyce structure. In this paper, we show how to use a Joyce structure to define a generating function which we call the $\tau$-function. When applied to the derived category of the resolved conifold, this reproduces the non-perturbative topological string partition function of [\textit{J.~Differential Geom.} \textbf{115} (2020), 395--435, arXiv:1703.02776]. In the case of the derived category of the Ginzburg algebra of the A$_2$ quiver, we obtain the Painlev{\'e} I $\tau$-function.}


\Classification{53C26; 53C28; 53D30; 34M55; 14N35}

\section{Introduction}

This paper is the continuation of a programme which attempts to encode the Donaldson--Thomas (DT) invariants of a CY$_3$ triangulated category $\cD$ in a geometric structure on the space of stability conditions $M=\Stab(\cD)$. The relevant geometry is a kind of non-linear Frobenius structure, and was christened a Joyce structure in \cite{RHDT2} in honour of the paper \cite{holgen} where the main ingredients were first discovered. In later work with Strachan \cite{Strachan}, it was shown that a~Joyce structure can be re-expressed in terms of a complex \hk structure on the total space of the tangent bundle $X=T_M$. An introduction to Joyce structures and their twistor spaces can be found in \cite{BJoy}.

The procedure for producing Joyce structures from DT invariants is conjectural, and requires solving a family of non-linear Riemann--Hilbert (RH) problems \cite{RHDT1} involving maps from the complex plane into a torus $(\bC^*)^n$ with prescribed jumps across a collection of rays. These problems are only defined if the DT invariants of the category satisfy a sub-exponential growth condition. So far there are no general results on existence or uniqueness of solutions. Nonetheless, there are several situations where it is possible to find natural solutions to the RH problems and explicitly describe the resulting Joyce structures.

It was discovered in \cite{con} that when $\cD$ is the 
derived category of coherent sheaves on the resolved conifold, the solutions to the associated RH problems can be repackaged in terms of a~single function, which can moreover be viewed as a non-perturbative A-model topological string partition function, since its asymptotic expansion coincides with the generating function for the Gromov--Witten invariants. This function was introduced in a rather ad hoc way however, and it was unclear how to extend its definition to more general settings.

 The aim of this paper is to formulate a general definition of such generating functions and study their properties. We associate to a Joyce structure on a complex manifold $M$, equipped with choices of certain additional data, a locally-defined function $\tau\colon X=T_M\to \bC^*$ which we call the $\tau$-function.
In the case of the derived category of the resolved conifold, and for appropriate choices of the additional data, the restriction of this $\tau$-function to a natural section $M\subset T_M$ coincides with the non-perturbative partition function obtained in \cite{con}.

The RH problems associated to the DT theory of the resolved conifold are rather special, in that the jumps across any two rays commute. A more representative class of examples can be obtained from the DT theory of CY$_3$ triangulated categories of class $S[A_1]$. These categories~${\cD=\cD(g,m)}$ are indexed by a genus $g\geq 0$ and a collection of pole orders~${m=\{m_1,\dots,m_l\}}$. When $l>0$, they can be defined using quivers with potential associated to triangulations of marked bordered surfaces \cite{L}, or via Fukaya categories of certain non-compact Calabi--Yau threefolds \cite{Smith}. The relevant threefolds $Y(g,m)$ are fibred over a~Riemann surface $C$, and are described locally by an equation of the form $y^2+uv=Q(x)$.

It was shown in \cite{BS} that when $l>0$ the space of stability conditions on the category $\cD(g,m)$, quotiented by the group of auto-equivalences,
 is the moduli space of pairs $(C,Q)$ consisting of a Riemann surface $C$ of genus $g$, equipped with a quadratic differential $Q$ with poles of order~${\{m_1,\dots,m_l\}}$ and simple zeroes. The associated DT invariants were shown to be counts of finite-length horizontal trajectories, as predicted by earlier work in physics \cite{GMN2,KLMVW}. These results have been extended by Haiden \cite{H} to the case $l=0$ involving holomorphic quadratic differentials.
 The general story described above then leads one to look for a natural Joyce structure on this space. In the case of differentials without poles this was constructed in \cite{CH}, and the generalisation to meromorphic differentials will appear in the forthcoming work \cite{Z}. The key ingredient in these constructions is the existence of isomonodromic families of bundles with connections.\looseness=1

In several examples, a non-perturbative completion of the $B$-model topological string partition function of the threefold $Y(g,m)$ is known to be related, via the Nekrasov partition function of the associated four-dimensional supersymmetric gauge theories of class $S[A_1]$, to an isomonodromic $\tau$-function \cite{BGT1,BGT2,BLMST}. It therefore becomes natural to try to relate the $\tau$-function associated to a Joyce structure of class $S[A_1]$ to an isomonodromic $\tau$-function. That something along these lines should be true was suggested by the work of Teschner and collaborators \cite{T1,T2}. The definition of the Joyce structure $\tau$-function was then reverse-engineered using the work of Bertola and Korotkin \cite{BK1} on the moduli-dependence of isomonodromic $\tau$-functions.

We shall treat one example of class $S[A_1]$ in detail below. It corresponds to taking $g=0$ and $m=\{7\}$. The resulting category $\cD(g,m)$ is the derived category of the CY$_3$ Ginzburg algebra associated to the A$_2$ quiver. The corresponding Joyce structure was constructed in~\cite{A2}. We show that with appropriate choices of the additional data the resulting Joyce structure $\tau$-function coincides with the Painlev{\'e} I $\tau$-function studied by Lisovyy and Roussillon~\cite{LR}.

 {\bf Plan of the paper.} We begin in Section \ref{general} by reviewing material from \cite{BJoy,Strachan}. We introduce the notion of a pre-Joyce structure on a complex manifold $M$ and the corresponding complex \hk structure on the total space $X=T_M$ of the tangent bundle. A Joyce structure is defined as a pre-Joyce structure with certain extra symmetries which are controlled by an additional structure on $M$ called a period structure.

In Section \ref{twistor}, we define the twistor space $p\colon Z\to \bP^1$ associated to a Joyce structure and recall some of its basic properties. We also give a brief preview of the definition of the $\tau$-function which is revisited in detail in Section \ref{tausection}.

In Section \ref{geometric}, we give a conjectural description of a class of Joyce structures relating to theories of class $S[A_1]$. The base $M$ of these structures parameterises Riemann surfaces of genus~$g$ equipped with a meromorphic quadratic differential having poles of fixed orders $\{m_1,\dots, m_l\}$. In the case $l=0$ of holomorphic differentials, these are constructed rigorously by a different method in \cite{CH}. The general case will be treated in \cite{Z}.

In Section \ref{twist}, we consider certain additional structures on the twistor space of a Joyce structure which are present in many examples, and which are relevant for the definition of the $\tau$-function. These are: (i) the structure of a cotangent bundle on the fibre $Z_0$, (ii) the existence of collections of preferred Darboux co-ordinates on the fibre $Z_1$, and (iii) a preferred choice of a Lagrangian submanifold in $Z_\infty$. We recall from \cite{BJoy} that the combination of (i) and (iii) gives rise to an integrable Hamiltonian system.

The definition of the $\tau$-function associated to a Joyce structure appears in Section \ref{tausection}.
 It depends on the choice of certain additional data, namely symplectic potentials on the twistor fibres $Z_0$, $Z_1$ and $Z_\infty$. We explain how these choices relate to the additional properties of the twistor space discussed in the previous section. We then show that when restricted to various loci the $\tau$-function produces generating functions for certain naturally associated symplectic maps. We also explain the relation with $\tau$-functions in the usual sense of Hamiltonian systems.

In Section \ref{uncoupled}, we consider the Joyce structure $\tau$-functions associated to uncoupled BPS structures. When restricted to a section of the projection $\pi\colon X\to M$ we show that our definition reproduces the $\tau$-functions defined in \cite{RHDT1}. In particular, this applies to the non-perturbative partition function of the resolved conifold computed in \cite{con}.

In Section \ref{a2},
 we consider the Joyce structure arising from the DT theory of the A$_2$ quiver. This was constructed in \cite{A2} using the monodromy map for the deformed cubic oscillator. We show that the Joyce structure $\tau$-function coincides with the Painlev{\'e} I $\tau$-function extended as a~function of monodromy exactly as described by Lisovyy and Roussillon \cite{LR}.

{\bf Conventions.} We work throughout in the category of complex manifolds and holomorphic maps. All symplectic forms, metrics, bundles, connections, sections etc.\ are holomorphic. The tangent bundle of a complex manifold $M$ is denoted $T_M$, and the derivative of a map of complex manifolds $f\colon M\to N$ is denoted $f_*\colon T_M\to f^*(T_N)$. The map $f$ is called {\'e}tale if $f_*$ is an isomorphism. We use the symbol $\cL$ to denote the Lie derivative.

\section{Joyce structures}
\label{general}

In this section, we introduce the geometric structures that will appear throughout the rest of the paper. They can be described either in terms of flat pencils of symplectic non-linear connections~\cite{RHDT2}, or via complex \hk structures as in \cite{Strachan}. Most of this material is standard in the twistor-theory literature, see, for example, \cite{CMN,DM}, and goes back to the work of Pleba{\'n}ski \cite{P}. We base our treatment on \cite{BJoy} to which we refer for further details.

\subsection{Pre-Joyce structures}
\label{firstone}\label{coords}

Let $\pi\colon X\to M$ be a holomorphic submersion of complex manifolds. There is a short exact sequence of vector bundles
\[
0\lra V(\pi)\lRa{i} T_X\lRa{\pi_*} \pi^*(T_M)\lra 0,
\]
where $V(\pi)=\ker(\pi_*)$ is the sub-bundle of vertical tangent vectors. Recall that a \emph{non-linear $($or Ehresmann$)$
connection} on $\pi$ is a splitting of this sequence, given by a map of bundles $h\colon \pi^*(T_M)\to T_X$ satisfying $\pi_*\circ h=1$.

Writing $H=\im(h)$ and $V=V(\pi)$, the tangent bundle of $X$ decomposes as a direct sum~${T_X=H\oplus V}$. We call tangent vectors and vector fields \emph{horizontal} or \emph{vertical} if they lie in~$H$ or $V$ respectively. A vector field $u\in H^0(M,T_M)$ can be lifted to a horizontal vector field~${h(u)\in H^0(X,T_X)}$ by composing the pullback $\pi^*(u)\in H^0(X,\pi^*(T_M))$ with the map $h$.

The connection $h$ is \emph{flat} if the following equivalent conditions hold:
\begin{itemize}\itemsep=0pt
\item[(i)] for every $x\in X$ there are local co-ordinates $(x_1,\dots, x_n)$ on $X$ at $x$, and $(y_1,\dots, y_d)$ on $M$ at $\pi(x)$, such that $x_i=\pi^*(y_i)$ and $h\bigl(\frac{\partial}{\partial y_i}\bigr)=\frac{\partial}{\partial x_i}$ for $1\leq i\leq d$,
\item[(ii)] the sub-bundle $H=\im(h)\subset T_X$ is involutive: $[H,H]\subset H$.
\end{itemize}

Consider the special case in which $\pi\colon X=T_M\to M$ is the total space of the tangent bundle of $M$. There is then a canonical isomorphism $\nu\colon \pi^*(T_M)\to V(\pi)$ identifying the vertical tangent vectors in the bundle with the bundle itself, and we set $v=i\circ\nu$
\begin{equation*}
\xymatrix@C=1em{
 0\ar[rr] && V(\pi) \ar[rr]^{i} &&T_X \ar[rr]^{\pi_*} &&\pi^*(T_M) \ar@/_1.8pc/[ll]_{h_\epsilon} \ar@/^1.8pc/[llll]^\nu \ar[rr] && 0. } \end{equation*}

Suppose that $M$ is equipped with a holomorphic symplectic form $\omega\in H^0\bigl(M,\wedge^2 T^*_M\bigr)$. Via the isomorphism $\nu$ we obtain a relative symplectic form $\Omega_\pi\in H^0\bigl(X,\wedge^2 T_{X/M}^*\bigr)$ which restricts to a~translation-invariant symplectic form $\omega_m$ on each fibre $X_m=T_{M,m}$.
We say that the connection~$h$ on $\pi$ is \emph{symplectic} if for any path $\gamma\colon [0,1]\to M$ the partially-defined parallel transport maps~${\operatorname{PT}_{\gamma}(t)\colon X_{\gamma(0)}\to X_{\gamma(t)}}$ take $\Omega_{\gamma(0)}$ to $\Omega_{\gamma(t)}$.

\begin{Definition}
A \emph{pre-Joyce structure} $(\omega,h)$ on a complex manifold $M$ consists of
\begin{itemize}\itemsep=0pt
\item[(i)] a symplectic form $\omega$ on $M$,
\item[(ii)] a {non-linear connection} $h$ on the tangent bundle $\pi\colon X=T_M\to M$,
\end{itemize}
such that for each $\epsilon\in \bC^*$ the connection $h_\epsilon=h+\epsilon^{-1}v$ is flat and symplectic.\end{Definition}

 Given local co-ordinates $(z_1,\dots,z_n)$ on $M$ there are associated linear co-ordinates $(\theta_1,\dots,\theta_n)$ on the tangent spaces $T_{M,p}$ obtained by writing a tangent vector in the form $\sum_i \theta_i\cdot {\partial}/{\partial z_i}$. We thus get induced local co-ordinates $(z_i,\theta_j)$ on the space $X=T_M$. We always assume that the co-ordinates $z_i$ are Darboux, in the sense that
\[
\omega = \frac{1}{2} \sum_{p,q} \omega_{pq} \cdot {\rm d}z_p\wedge {\rm d}z_q,
\]
with $(\omega_{pq})_{p,q=1}^n$ a constant skew-symmetric matrix. We denote by $(\eta_{pq})_{p,q=1}^n$ the inverse matrix.

We can express them maps $v$ and $h$ in the form
\begin{equation}
\label{above}
v_i=v\left(\frac{\partial}{\partial z_i}\right)= \frac{\partial}{\partial \theta_i}, \qquad h_i=h\left(\frac{\partial}{\partial z_i}\right)= \frac{\partial}{\partial z_i} + \sum_{p,q} \eta_{pq} \cdot \frac{\partial W_i}{\partial \theta_p} \cdot \frac{\partial}{\partial \theta_q},
\end{equation}
for locally-defined functions $W_i$ on $X$.
The connection $h_\epsilon$ is flat precisely if $\big[h_i+\epsilon^{-1} v_i,\allowbreak {h_j+\epsilon^{-1} v_j}\big]=0$ for all $i$, $j$. A short calculation shows that this holds for all $\epsilon\in \bC^*$ precisely if
\begin{gather}
 \frac{\partial}{\partial \theta_k}\left(\frac{\partial W_i}{\partial \theta_j}-\frac{\partial W_j}{\partial \theta_i}\right)=0, \qquad
 \frac{\partial}{\partial \theta_k}\left(\frac{\partial W_i}{\partial z_j}-\frac{\partial W_j}{\partial z_i }-\sum_{p,q} \eta_{pq} \cdot \frac{\partial W_i}{\partial \theta_p} \cdot \frac{\partial W_j}{\partial \theta_q}\right)=0.\label{jude}
 \end{gather}

\subsection{Joyce structures}
\label{further}

 A Joyce structure is a pre-Joyce structure with certain extra symmetries. These symmetries are controlled by an additional structure called a period structure.

 \begin{Definition}A \emph{period structure} on a complex manifold $M$ is given by a collection of distinguished charts $\phi_i\colon U_i\to \bC^n$ whose transition functions are maps of the form $z_i\mapsto \sum_j a_{ij} z_j$ with $(a_{ij})_{i,j=1}^n\in \GL_n(\bZ)$. The local co-ordinate systems $(z_1,\dots,z_n)$ associated to distinguished charts will be called \emph{integral linear co-ordinates}.
\end{Definition}

A Joyce structure on a complex manifold $M$ consists of a period structure and a pre-Joyce structure satisfying several compatibility relations.
For the precise definition, we refer the reader to \cite[Section 3]{BJoy}. Here we will just give a description in terms of local co-ordinates.
Henceforth, ${(z_1,\dots,z_n)}$ will always denote a local system of co-ordinates on $M$ which are integral linear for the period structure underlying our Joyce structure. We then consider the associated co-ordinates $(z_i,\theta_j)$ on $X=T_M$ as above.
There is a well-defined vector field on $M$
\[
Z=\sum_i z_i\cdot \frac{\partial}{\partial z_i}.
\]
We also consider the corresponding lifted vector field on $X$ given by
\begin{equation} \label{e}E=\sum_i z_i\cdot \frac{\partial}{\partial z_i}.\end{equation}
We refer to either of these vector fields as \emph{Euler vector fields}. We say that the Joyce structure is \emph{homogeneous} if they generate $\bC^*$-actions on $M$ and $X$ respectively. For simplicity, we will often assume that this is the case in what follows.

It was shown in \cite{BJoy} that in the case of a Joyce structure one can rewrite the vector fields~\eqref{above} in the form
\[
v_i=v\left(\frac{\partial}{\partial z_i}\right)= \frac{\partial}{\partial \theta_i}, \qquad h_i=h\left(\frac{\partial}{\partial z_i}\right)= \frac{\partial}{\partial z_i} + \sum_{p,q} \eta_{pq} \cdot \frac{\partial^2 W}{\partial \theta_i\partial \theta_p} \cdot \frac{\partial}{\partial \theta_q},
\]
where $W$ is a single locally-defined function on $X$ called the \emph{Pleba{\'n}ski function}. This can be normalised by requiring
\[
 W(z_1,\dots,z_n, 0, \dots, 0)=0=\frac{\partial W}{\partial \theta_i}(z_1,\dots, z_n,0,\dots,0).
 \]
The flatness conditions \eqref{jude} then become \emph{Pleba{\'n}ski's second heavenly equations}
 \[
\frac{\partial^2 W}{\partial \theta_i \partial z_j}-\frac{\partial^2 W}{\partial \theta_j \partial z_i }=\sum_{p,q} \eta_{pq} \cdot \frac{\partial^2 W}{\partial \theta_i \partial \theta_p} \cdot \frac{\partial^2 W}{\partial \theta_j \partial \theta_q}.
\]

The first axiom of a Joyce structure is that the inverse of the symplectic form satisfies $\eta_{pq}\in 2\pi {\rm i} \bZ$. The other axioms are equivalent to the following symmetry properties of the function $W$:
\begin{gather}
\label{gove2}
\frac{\partial^2 W}{ \partial \theta_i \partial \theta_j}(z_1,\dots,z_n,\theta_1+2\pi {\rm i} k_1,\dots,\theta_n+2\pi {\rm i}k_n)=\frac{\partial^2 W}{\partial \theta_i \partial \theta_j }(z_1,\dots,z_n,\theta_1,\dots, \theta_n), \\
 \label{reeves2}W(\lambda z_1,\dots,\lambda z_n,\theta_1,\dots, \theta_n)=\lambda^{-1} \cdot W(z_1,\dots,z_n,\theta_1,\dots, \theta_n), \\
W(z_1,\dots,z_n,-\theta_1,\dots, -\theta_n)=-W(z_1,\dots,z_n,\theta_1,\dots, \theta_n),\nonumber
\end{gather}
where $(k_1,\dots, k_n)\in \bZ^n$ and $\lambda\in \bC^*$.

\subsection[Complex hyperk\"ahler structure]{Complex \hk structure}

We recall the following definition from \cite{BJoy}.

\begin{Definition}
A \emph{complex \hk structure} $(g,I,J,K)$ on a complex manifold $X$ consists of a non-degenerate symmetric bilinear form
$g\colon T_X \tensor T_X \to \O_X,$
together with endomorphisms~${I,J,K\in \End(T_X)}$ satisfying the quaternion relations
\[
I^2=J^2=K^2=IJK=-1,
\]
which preserve the form $g$, and which are parallel with respect to the Levi-Civita connection.
\end{Definition}

A pre-Joyce structure on a complex manifold $M$ induces a complex \hk structure on the total space $X=T_M$. The action of the quaternions is given by
\begin{gather*}
I(h_i) ={\rm i}\cdot h_i,\qquad I(v_i) =-{\rm i}\cdot v_i,\qquad J(h_i) =-v_i,\qquad J(v_i) =h_i,\\
 K(h_i) ={\rm i} v_i, \qquad
   K(v_i) ={\rm i} h_i,
 \end{gather*}
and the metric $g$ is defined by
\[
g(h_i,h_j)=0, \qquad g(h_i,v_j)=\half \omega_{ij}, \qquad g(v_i,v_j)=0.
\]

There are associated closed 2-forms on $X$ given by
\[
\Omega_I(w_1,w_2)=g(I(w_1),w_2),\qquad \Omega_{\pm}(w_1,w_2)=g((J\pm {\rm i}K)(w_1),w_2).
\]
Explicitly, we have
\begin{gather}
\Omega_+=\half\cdot \sum_{p,q} \omega_{pq} \cdot {\rm d}z_p \wedge {\rm d}z_q, \qquad 2{\rm i} \Omega_I=-\sum_{p,q} \omega_{pq} \cdot {\rm d} z_p \wedge {\rm d} \theta_q, \nonumber\\
\Omega_-=\half\cdot \sum_{p,q}\omega_{pq} \cdot {\rm d}\theta_p\wedge {\rm d}\theta_q+\sum_{p,q} \frac{\partial^2 W}{\partial \theta_p \partial\theta_q} \cdot {\rm d} \theta_p \wedge {\rm d} z_q +\sum_{p,q} \frac{\partial^2 W}{\partial z_p \partial \theta_q} \cdot {\rm d}z_p\wedge {\rm d}z_q.\label{w}
\end{gather}

There are identities
\begin{equation}
\label{scales}
\Lie_E(\Omega_+)=2\Omega_+, \qquad \Lie_E(\Omega_I)=\Omega_I, \qquad \Lie_E(\Omega_-)=0,
\end{equation}
which follow immediately from \eqref{e} and the homogeneity property \eqref{reeves2}.

\begin{Remark}So as to be able to include certain interesting examples such as those discussed in Section \ref{geometric} below, it is often useful to weaken the axioms of a Joyce structure to allow the connections $h_\epsilon\colon \pi^*(T_M)\to T_X$ to have poles. When expressed in terms of local co-ordinates as above, this just means that the function $W(z,\theta)$ is meromorphic. We will refer to the resulting structures as \emph{meromorphic Joyce structures}.
 \end{Remark}

\section{Twistor space}
\label{twistor}

The geometry of a Joyce structure is often clearer when viewed through the lens of the associated twistor space \cite{BE,HLMR}. In this section, we define the twistor space and give some basic properties. We also give a preview of the definition of the $\tau$-function which is the main topic of this paper.

\subsection{Definition of twistor space}

Let $(\omega,h)$ be a pre-Joyce structure on a complex manifold $M$. There is then a pencil of flat, non-linear connections $h_\epsilon=h+\epsilon^{-1}v$ on the tangent bundle $\pi\colon X=T_M\to M$.
For a given point $\epsilon \in \bC^*\cup\{\infty\}$, we can consider the sub-bundle $H(\epsilon):=\im(h_\epsilon)\subset T_X$.
Since $h_\epsilon$ is flat this sub-bundle is involutive: $[H(\epsilon),H(\epsilon)]\subset H(\epsilon)$.
The twistor fibre $Z_\epsilon$ is defined to be the space of leaves of the associated foliation on $X$. We denote by $q_\epsilon\colon X\to Z_\epsilon$ the quotient map.

More globally, we consider the product $\bP^1\times X$ and the projection maps
$\bP^1\lLa{\pi_1} \bP^1\times X\lRa{\pi_2} X$.
As $\epsilon \in \bP^1$ varies, the sub-bundles $H(\epsilon)\subset T_M$ combine to give a sub-bundle $H\subset \pi_2^*(T_X)$. We can then view this as a sub-bundle of $T_{\bP^1\times X}$ via the canonical decomposition $T_{\bP^1\times X}=\pi_1^*(T_{\bP^1})\oplus \pi_2^*(T_X)$. This sub-bundle is involutive, and we define the \emph{twistor space} $Z$ to be the space of leaves of the associated foliation. We denote by $q\colon \bP^1\times X\to Z$ the quotient map.

There is an induced projection $p\colon Z\to \bP^1$ satisfying $p\circ q=\pi_1$. The fibre $p^{-1}(\epsilon)$ over a~point~${0\neq \epsilon\in \bP^1}$ coincides with the space $Z_\epsilon$ defined before. The fibre $Z_0=p^{-1}(0)$ is the space of leaves of the distribution $V=\ker(\pi_*)\subset T_X$ and is therefore identified with $M$.
The situation is summarised in the following diagram:
\begin{equation*}
\xymatrix@C=1.7em{ X \ar@{^{(}->}[rr] \ar[d]_{q_\epsilon} && \bP^1\times X \ar[d]^{q}\ar@/^2.5pc/[dd]^{\pi_1} \\
Z_\epsilon \ar@{^{(}->}[rr]\ar[d] && Z\ar[d]^{p} \\ \{\epsilon\} \ar@{^{(}->}[rr]&& \bP^1 }
\end{equation*}
in which the horizontal arrows are the obvious closed embeddings.

\begin{Remark}
Unfortunately, to obtain a well-behaved twistor space $p\colon Z\to \bP^1$ we cannot in general just take the space of leaves of the foliation on $\bP^1\times X$. Rather, we should consider the holonomy groupoid, which leads to the analytic analogue of a Deligne--Mumford stack \cite{Moer}. We will completely ignore these subtleties here, and essentially pretend that $Z$ is a complex manifold.
In fact, for what we do here, nothing useful would be gained by a more abstract point of view, because we are only really using the twistor space as a useful and suggestive shorthand. All statements we make about the space $Z$ can be easily translated into statements only involving objects on $X$. For example, a symplectic form on the twistor fibre $Z_\epsilon$ is nothing but a closed 2-form on $X$ whose kernel coincides with the sub-bundle $H(\epsilon)\subset T_X$. Similarly, an {\'e}tale map from the twistor fibre $Z_\epsilon$ to some complex manifold $Y$ is just a holomorphic map~${f\colon X\to Y}$ such that $\ker(f_*)=H(\epsilon)$.
\end{Remark}

Recall the complex \hk structure $(g,I,J,K)$ on $X$ and the associated closed 2-forms~$\Omega_\pm$ and $\Omega_I$. When $\epsilon \in \bC^*\cup\{\infty\}$ a simple calculation shows that
\[
H(\epsilon)=\im\big(h+\epsilon^{-1}v \big)=\ker\big(\epsilon^{-2} (J+{\rm i}K)+2{\rm i} \epsilon^{-1} I + (J-{\rm i}K)\big)\subset T_X.
\]
It follows that there is a symplectic form $\Omega_\epsilon$ on the twistor fibre $Z_\epsilon$ such that
\begin{equation}
\label{qe}q_\epsilon^*(\Omega_\epsilon)=\epsilon^{-2}\Omega_++2{\rm i}\epsilon^{-1} \Omega_I+\Omega_-.\end{equation}

More globally, the right-hand side of \eqref{qe} defines a twisted relative symplectic form on the twistor space $p\colon Z\to \bP^1$. This is a section of the bundle $ p^*(\O(2))\tensor \wedge^2  T_{Z/\bP^1}^*$ and restricts to give a 2-form $\Omega_\epsilon$ on each twistor fibre $Z_\epsilon$ which is well defined up to multiplication by a nonzero constant.
 When $\epsilon\neq 0$, we can fix this scale by imposing \eqref{qe}.
We specify the scale on the fibre~${Z_0=M}$ by taking $\Omega_0=\omega$. Thus
$
q_0^*(\Omega_0)=\Omega_+$, $ q_\infty^*(\Omega_\infty)=\Omega_-$.
In what follows, we will generally try to distinguish between objects living on a twistor fibre $Z_\epsilon$ and their pullbacks via the quotient map $q_\epsilon\colon X\to Z_\epsilon$. When writing formulae, however, this distinction tends to be a~distraction and we will sometimes drop it.

\subsection{Twistor space of a Joyce structure}
\label{ab}
Let us now consider a homogeneous Joyce structure on a complex manifold $M$, and the associated twistor space $p\colon Z\to \bP^1$. In particular, there is a $\bC^*$-action on $M$ generated by the vector field~$Z$, and a $\bC^*$-action on $X$ generated by the lifted vector field $E$. We then consider the diagonal $\bC^*$-action on $\bP^1\times X$, where the action on $\bP^1$ is the standard one rescaling $\epsilon$ with weight 1. It follows from the conditions \eqref{scales} that this action descends along the quotient map $q\colon \bP^1\times X\to Z$ to give a $\bC^*$-action on $Z$.

The map $p\colon Z\to \bP^1$ intertwines the $\bC^*$-actions, and it follows that there are essentially three distinct twistor fibres: $Z_0$, $Z_1$ and $Z_\infty$, of which $Z_0=M$ is the base of the Joyce structure. The $\bC^*$-action on $Z$ restricts to $\bC^*$-actions on the fibres $Z_0$ and $Z_\infty$. The relations \eqref{scales} show that the symplectic form $\Omega_0$ on $Z_0$ is homogeneous of weight 2, whereas $\Omega_\infty$ on $Z_\infty$ is $\bC^*$-invariant.

We introduce 1-forms on $X$ via the formulae
\begin{equation*} 
\alpha_+= i_E(\Omega_+), \qquad \alpha_I = i_E(\Omega_I), \qquad \alpha_-= i_E(\Omega_-).
\end{equation*}
The relations \eqref{scales} together with the Cartan formula imply that
\begin{equation}
\label{scales2}
{\rm d} \alpha_+=2\Omega_+, \qquad {\rm d}\alpha_I=\Omega_I, \qquad {\rm d}\alpha_-=0.
\end{equation}
The closed form $\alpha_-$ was discussed in detail in \cite[Section 5.1]{BJoy}. The forms $\alpha_+$ and $\alpha_I$ will play an important role below. In terms of local co-ordinates, we have
\begin{equation}\label{snowy}
\alpha_+= \sum_{p,q} \omega_{pq} \cdot z_p  {\rm d}z_q, \qquad 2{\rm i}\alpha_I= -\sum_{p,q} \omega_{pq} \cdot z_p  {\rm d} \theta_q.
\end{equation}

 The form $\alpha_+$ clearly descends to the twistor fibre $Z_0=M$.
Thus we can write $\alpha_+=q_0^*(\alpha_0)$, with $\alpha_0$ a form on $Z_0$, and we see that the symplectic form $\Omega_0$ is exact, with canonical symplectic potential $\half \alpha_0$. Moreover, this potential is homogeneous of weight 2 for the $\bC^*$-action on $Z_0$. The form $\alpha_I$ similarly provides a canonical symplectic potential for the symplectic form~$\Omega_I$ on~$X$, and is homogeneous of weight~1.

\subsection[Preview of the tau-function]{Preview of the $\boldsymbol{\tau}$-function}

 We now give a preview of the definition of the $\tau$-function associated to a Joyce structure. Consider the identity of closed 2-forms on $X=T_M$
 \begin{equation}
 \label{ide2}q_1^*(\Omega_1)=\Omega_++2{\rm i} \Omega_I + \Omega_-\end{equation}
obtained by setting $\epsilon=1$ in \eqref{qe}.
Let us choose locally-defined 1-forms $\Theta_I$ and $\Theta_\pm$ on $X$ satisfying
\[
{\rm d}\Theta_I=\Omega_I, \qquad {\rm d}\Theta_\pm=\Omega_\pm.
\]
It is natural to assume that the forms $\Theta_+$ and $\Theta_-$ descend to the twistor fibres $Z_0$ and $Z_\infty$, respectively, and to require the homogeneity properties
\[
\cL_E(\Theta_+)=2\Theta_+, \qquad \cL_E(\Theta_I)=\Theta_I, \qquad \cL_E(\Theta_-)=0.
\]

Let us also choose a locally-defined 1-form $\Theta_1$ on the twistor fibre $Z_1$ satisfying ${\rm d}\Theta_1=\Omega_1$.
Then the corresponding \emph{$\tau$-function} is the locally-defined function on $X$ uniquely specified up to multiplication by constants by the relation
\begin{equation}
\label{tau3}
{\rm d}\log(\tau)=\Theta_++2{\rm i}\Theta_I +\Theta_--q_1^*(\Theta_1).\end{equation}
Note that \eqref{ide2} implies that the right-hand side of \eqref{tau3} is closed.

The definition \eqref{tau3} is of course vacuous without some prescription for the symplectic potentials appearing. From the discussion of Section \ref{ab}, there are canonical choices $\Theta_+=\alpha_+$ and~${\Theta_I=\alpha_I}$, so what remains is to specify $\Theta_1$ and $\Theta_\infty$. However we will see below that even for~$\Theta_I$ and $\Theta_+$ it may be useful to make different choices. In this paper, we will only make partial progress towards resolving these issues. Before revisiting them in more detail in Section~\ref{tausection}, we will first introduce an important class of examples of Joyce structures, and use them to motivate some relevant geometric properties of the twistor space.

\section[Joyce structures of class \protect{S[A\_1]}]{Joyce structures of class $\boldsymbol{S[A_1]}$}
\label{geometric}
As discussed in the introduction, there is a class of meromorphic Joyce structures which are related to supersymmetric theories of class $S[A_1]$. The base $M$ parameterizes pairs consisting of a Riemann surface equipped with a quadratic differential having simple zeroes and poles of fixed orders. In the case of holomorphic quadratic differentials, the Joyce structure was constructed in \cite{CH}. The construction in the meromorphic case will appear in \cite{Z}. In this section, we give a sketch of a `quick and dirty' approach to these Joyce structures which will be enough to understand some of their general features, and which may be of independent interest.

\subsection{Quadratic differentials}

We begin by fixing the data of a genus $g\geq 0$ and a collection of integers $m=\{m_1,\dots,m_l\}$ with all $m_i\geq 2$.
There is then a space
$\Quad(g,m)$ parameterizing pairs $(C,Q_0)$, where $C$ is a~compact, connected Riemann surface of genus $g$, and $Q_0$ is a meromorphic section of $\omega_C^{\tensor 2}$ with simple zeroes, and
 $l$ poles $x_i\in C$ with multiplicities $m_i$. For the basic properties of these spaces, we refer the reader to \cite{BS}. We always assume that
\[
k:=6g-6+\sum_{i=1}^l (m_i+1)>0.
\]
The space $\Quad(g,m)$ is then a non-empty complex orbifold of dimension $k$.

Given a point $(C,Q_0)\in \Quad(g,m)$ there is a spectral curve $p\colon \Sigma\to C$ branched at the zeroes and odd order poles of $Q_0$. Locally it is given by writing $y^2=Q_0(x)$. There is a covering involution $\sigma\colon \Sigma\to \Sigma$ and a canonical meromorphic 1-form $\lambda=y {\rm d} x$ satisfying $\lambda^{\tensor 2}=p^*(Q_0)$.
We define $\Sigma^0\subset \Sigma$ to be the complement of the poles of $\lambda$.

We consider the homology group $H_1\bigl(\Sigma^0,\bZ\bigr)^-$. The superscript signifies anti-invariance for the covering involution: we consider only classes satisfying $\sigma_*(\gamma)=-\gamma$. A calculation shows that $H_1\bigl(\Sigma^0,\bZ\bigr)^-\isom \bZ^{\oplus k}$ is free of rank $k$.
Given a basis $(\gamma_1,\dots,\gamma_k)\subset H_1\bigl(\Sigma^0,\bZ\bigr)^-$ we define
\begin{equation}\label{zzz}z_i=\int_{\gamma_i} \lambda\in \bC.\end{equation}

As the point $(C,Q_0)\in \Quad(g,m)$ varies the homology groups $H_1\bigl(\Sigma^0,\bZ\bigr)^-$ fit together to form a local system over $\Quad(g,m)$. Transporting the basis elements $\gamma_i$ to nearby points, the resulting functions $(z_1,\dots,z_k)$ form local co-ordinates on $\Quad(g,m)$. In particular, the tangent space at a point of $\Quad(g,m)$ is naturally identified with the cohomology group $H^1(\Sigma^0,\bC)^-$.

The inclusion $i\colon \Sigma^0\hookrightarrow \Sigma$ induces a surjection $i_*\colon H_1\bigl(\Sigma^0,\bZ\bigr)^-\to H_1(\Sigma,\bZ)^-$.
The intersection form on $H_1(\Sigma,\bZ)$ pulls back to an integral skew-symmetric form $\<-,-\>$ on \smash{$H_1\bigl(\Sigma^0,\bZ\bigr)^-$}. This induces a Poisson structure on $\Quad(g,m)$ satisfying
\[
\{z_i,z_j\}=2\pi {\rm i} \<\gamma_i,\gamma_j\>.
\]

The kernel of $i_*$ is spanned (at least rationally) by classes $\pm \beta_i\in H_1\bigl(\Sigma^0,\bZ\bigr)^-$ defined up to sign by the difference of small loops around the two inverse images of an even-order pole of $Q_0$.
The residue of $Q_0$ at such a pole $x_i\in C$ is defined to be the period
\begin{equation}\label{res2}\res_{x_i}(Q_0)=\pm \int_{\beta_i}\lambda\in \bC.\end{equation}
It is well defined up to sign and is a Casimir for the above Poisson structure. Fixing these residues locally cuts out a symplectic leaf in $\Quad(g,m)$ of dimension $n=2d$, where
\begin{equation}\label{d} d=3g-3+ \sum_{i=1}^l  \left\lfloor \half (m_i+1)\right\rfloor.\end{equation}
There are 10 choices of the data $(g,m)$ with $g=0$ for which $n=2$, and these correspond to the Painlev{\'e} equations (see, for example, \cite[Section~2.1]{BLMST}).

We will consider the particular symplectic leaf
\[
M=M(g,m)\subset \Quad(g,m)\]
obtained by requiring all residues \eqref{res2} at even order poles to be zero. Note that a pole of order~2 with zero residue is actually a pole of order 1.

The tangent space to $M$ at a point $(C,q)$ is identified with the cohomology group $H^1(\Sigma,\bC)^-$. Choosing a basis $(\gamma_1,\dots, \gamma_n)\subset H^1(\Sigma,\bZ)^-$ gives local co-ordinates $(z_1,\dots,z_n)$ defined by the formula \eqref{zzz}.
These preferred local co-ordinates define a period structure on $M$. The inverse to the Poisson structure \eqref{res2} defines a symplectic form $\omega$ on $M$.

The remaining data we need to define a (meromorphic) Joyce structure on $M$ is a pencil of flat, symplectic, non-linear (meromorphic) connections $h_\epsilon$ on the tangent bundle $\pi\colon X=T_M\to M$. Recall that the period structure defines a subsheaf $T_M^{\bZ}\subset T_M$ and consider the quotient
\[
 X^\hash=T_M^{\hash}=T_M/(2\pi {\rm i}) T_M^{\bZ}.
 \] We aim to construct a non-linear connection on the projection $\pi\colon X^\hash\to M$. This then pulls back to give a connection on $\pi\colon X\to M$ which moreover automatically satisfies the periodicity condition \eqref{gove2}. Note that the local co-ordinates on $X^\hash$ are $(z_i,\xi_j)$ where $\xi_j={\rm e}^{\theta_j}$.

\subsection{Pencils of projective structures}
We will construct a pencil of non-linear (meromorphic) connections on the projection ${\pi\colon\! X^\hash\!\to\! M}$ by associating to a generic point of the space $X^\hash$ a pencil of projective structures with apparent singularities. The required connections will then be obtained by considering isomonodromic deformations of these projective structures.

We refer to \cite{AB} for a review of projective structures on Riemann surfaces. These objects can be equivalently described as $\PGL_2(\bC)$-opers. We will consider meromorphic projective structures on a compact Riemann surface $C$ of the form
\begin{equation}\label{oper}f''(x)=Q(x,\epsilon) \cdot f(x), \qquad Q(x,\epsilon)=\epsilon^{-2}\cdot {Q_0(x)}+\epsilon^{-1}\cdot Q_1(x) + Q_2(x),\end{equation}
where $\epsilon \in \bC^*$ and $Q_i(x)$ are meromorphic functions.
More invariantly, $Q_0= Q_0(x)  {\rm d}x^{\tensor 2}$ and~${Q_1= Q_1(x)  {\rm d}x^{\tensor 2}}$ are meromorphic quadratic differerentials on $C$, and $Q_2(x)$ represents a meromorphic projective structure.

We will require that the pair $(C,Q_0)$ defines a point of the space $M$.
At a point ${x_i\in C}$ where~$Q_0$ has a pole of order $m_i$ we will insist that $Q_1$ and $Q_2$ have poles of order at most $\bigl\lfloor \half (m_i+1)\bigr\rfloor$. We will allow $Q_1$ and $Q_2$ to have poles at exactly $d$ other points $q_i\in C$, with~$d$ given by \eqref{d}. At these points $Q_1$, $Q_2$ will be required to have the leading-order behaviour
\[
Q_1(x)= \frac{p_i}{x-q_i} + r_i + O(x-q_i), \qquad Q_2(x)=\frac{3}{4(x-q_i)^2} + \frac{s_i}{x-q_i} + u_i +O(x-q_i),
\]
with $p_i,r_i,s_i,u_i\in \bC$.
We will then insist that for all $\epsilon\in \bC^*$ the equation \eqref{oper} has apparent singularities at the points $x=q_i$, i.e., that the monodromy of the associated linear system is trivial asd an element of $\PGL_2(\bC)$.
This is equivalent to the condition
\[
	\bigl(\epsilon^{-1} p_i+s_i\bigr)^2=\epsilon^{-2}Q_0(q_i) +\epsilon^{-1} r_i+u_i,
\]
 for all $\epsilon\in \bC^*$ (see, e.g., \cite[Lemma 2.1]{A2}), and hence to the equations \begin{equation}
\label{conditions}p_i^2=Q_0(q_i), \qquad r_i=2p_is_i, \qquad u_i=s_i^2.\end{equation}
The first of these relations shows that the pair $(q_i,p_i)$ defines a point of the double cover $\Sigma$ associated to the quadratic differential $(C,Q_0)\in M$.

The condition on the pole orders of $Q_1$ at the points $x_i$, together with the equations \eqref{conditions}, ensures that the anti-invariant differential $Q_1/\sqrt{Q_0}$ on the double cover $\Sigma$ has simple poles at the points $(q_i,\pm p_i)$ with residues $\pm 1$ and no other poles. It follows that the expressions
\begin{equation}
 \label{xi2}
\xi_i=\exp\left(-\int_{\gamma_i} \frac{Q_1}{2\sqrt{Q_0}} \right)\in \bC^*\end{equation}
are well defined.

\subsection{Conjectural Joyce structure}
Let $\cO=\cO(g,m)$ denote the space of data $(C,Q_0,Q_1,Q_2)$ satisfying the conditions described in the previous subsection. There is a diagram of maps
\[
\xymatrix@C=1.7em{ \cO\ \ar[rr]^{\alpha} \ar[dr]_{\pi} && X^\hash\ar[dl]^{\pi} \\ & M,}
\]
where $\pi\colon \cO\to M$ sends a point $(C,Q_0,Q_1,Q_2)\in \cO$ to the point $(C,Q_0)\in M$, and $\alpha$ sends a~point $(C,Q_0,Q_1,Q_2)\in \cO$ to the point of $X^\hash$ whose local co-ordinates $(z_i,\xi_j)$ are given by the formulae \eqref{zzz} and \eqref{xi2}.

\begin{Conjecture}\label{firstguess}
The map $\alpha\colon \cO\to X^\hash$ is generically {\'e}tale.
\end{Conjecture}

\begin{proof}[Sketch proof] We must show that at a generic point of $\cO$ the functions $(z_i,\xi_j)$ give local co-ordinates. We know that the co-ordinates $z_i$ locally determine $(C,Q_0)$. The next step is to show that the numbers $\xi_i\in \bC^*$ locally determine $Q_1$, at least generically. Consider the line bundle
\[
L=\O_{\Sigma}\Biggl(\sum_{i=1}^d (q_i,p_i)-\sum_{i=1}^d(q_i,-p_i)\Biggr).
\]
 Since the 1-form $\eta=Q_1/\sqrt{Q_0}$ has simple poles with residues $\pm 1$ at the points $(q_i,\pm p_i)$, it follows that the expression $\partial={\rm d}-\eta$ defines an anti-invariant connection on $L$. Moreover, $\xi_i^2\in \bC^*$ is the holonomy of this connection around the cycle $\gamma_i\in H_1(\Sigma,\bZ)^-$. The abelian Riemann--Hilbert correspondence shows that these holonomies uniquely determine $(L,\partial)$. But now, by a version of the Jacobi inversion theorem, generically the bundle $L$ uniquely determines the points $(q_i,p_i)$.

The final step is to show that the meromorphic projective structure $Q_2$ is uniquely determined by the other data. Any two choices differ by a quadratic differential having poles of order at most $\bigl\lfloor \half (m_i+1)\bigr\rfloor$ at the points $x_i$, and by \eqref{conditions}, regular and vanishing at the points $q_i$. Such an object is a section of the line bundle
\[
\omega_C^{\tensor 2}\left(\sum_{i=1}^l \left\lceil \frac{m_i}{2} \right\rceil\cdot x_i -\sum_{j=1}^{k} (q_i,p_i)\right),
\]
 which has Euler characteristic $d-d=0$. Thus for general positions of the points $q_i\in C$, the projective structure $Q_2$ is uniquely determined.
\end{proof}

For each $\epsilon\in \bC^*$, we can attempt to define a non-linear connection $h_\epsilon$ on the projection $\cO\to M$ by considering isomonodromic deformations. That is, as $(C,Q_0)$ varies, we specify the variation of $(Q_1,Q_2)$ by insisting that the generalised monodromy of \eqref{oper} is constant. Unfortunately, the existence of such an isomonodromy connection has not yet
been established in the literature.

\begin{Conjecture}\label{secondguess}\quad
\begin{itemize}\itemsep=0pt
 \item[$(i)$]
For each $\epsilon\in \bC^*$, there is a flat, meromorphic connection $g_\epsilon$ on the projection $\pi\colon \cO\to M$ whose leaves define deformations of the projective structure \eqref{oper} with constant generalised monodromy.

\item[$(ii)$]
There exist meromorphic connections $h_\epsilon$ on the projection $\pi\colon X^\hash\to M$ whose pullback via the map $\alpha\colon \cO\to X^\hash$ are the connections $g_\epsilon$ of part $(i)$.

\item[$(iii)$] There is a meromorphic Joyce structure on $M$ obtained by combining the period structure and symplectic form on $M$ defined above with the connections $h_\epsilon$ of part $(ii)$.
\end{itemize}
\end{Conjecture}

Note that if Conjecture \ref{firstguess} holds, then there can be at most one meromorphic connection on $\pi\colon X^\hash\to M$ which lifts to a given meromorphic connection on $\pi\colon \cO\to M$.
We note that Conjecture \ref{secondguess} is known to hold for some specific choices of the data $(g,m)$. A different construction of the Joyce structures of Conjecture \ref{secondguess} will be given in \cite{Z} using bundles with connection rather than projective structures with apparent singularities.

\section{Further geometry of twistor space}
\label{twist}

 In this section, we discuss further structures on the twistor spaces of a Joyce structure which relate to choices of symplectic potentials on the twistor fibres $Z_0$, $Z_1$ and $Z_\infty$, and are highly relevant to the definition of the $\tau$-function in Section \ref{tausection}. We try to explain how these structures appear naturally in the conjectural examples of class $S[A_1]$ of Section \ref{geometric}, although many details are still missing. It would be interesting to study these structures in further examples, and examine the extent to which they can be defined in general.

\subsection[Cotangent bundle structure on M=Z\_0]{Cotangent bundle structure on $\boldsymbol{M=Z_0}$}
\label{cot}

Let $M$ be a complex manifold with a holomorphic symplectic form $\omega\in H^0\bigl(M,\wedge^2 T_M^*\bigr)$.
By a \emph{cotangent bundle structure} on $M$ we mean the data of a complex manifold $B$ and an open embedding $M\subset T_B^*$, such that $\omega$ is the restriction of the canonical symplectic form on $T_B^*$. We denote by $\rho\colon M\to B$ the induced projection map, and by $\lambda\in H^0(M,T^*_M)$ the restriction to $M$ of the canonical Liouville 1-form on $T_B^*$.

Given local co-ordinates $(t_1,\dots,t_d)$ on $B$, there are induced linear co-ordinates $(s_1,\dots, s_d)$ on the cotangent spaces $T_{B,b}^*$ obtained by writing a 1-form as $\sum_i s_i {\rm d}t_i$. In the resulting co-ordinates~${(s_i,t_i)}$ on $M$, we have
$
\omega=\sum_i {\rm d}t_i\wedge {\rm d}s_i$, $ \lambda=\sum_i s_i {\rm d}t_i$.

Note that $d(-\lambda)=\omega$. When $M=Z_0$ is the base of a homogeneous Joyce structure there is a~$\bC^*$-action on $M$ satisfying $\cL_E(\omega)=2\omega$. It is then natural to seek a cotangent bundle structure~${M\subset T_B^*}$ such that $\cL_E(\lambda)=2\lambda$. Note that this implies that the $\bC^*$-action on $M$ preserves the distribution of vertical vector fields for $\rho$, since this coincides with the kernel of $\lambda$. It follows that there is a $\bC^*$-action on $B$, and that the $\bC^*$-action on $M$ is given by the combining the induced action on $T_B^*$ with a rescaling of the fibres with weight 2.

Consider the conjectural Joyce structures of class $S[A_1]$ of Section \ref{geometric}, and assume for simplicity that $l=0$. Then $M$ parameterises pairs $(C,Q_0)$ consisting of a Riemann surface of genus $g$ equipped with a holomorphic quadratic differential $Q_0\in H^0\bigl(C,\omega_C^{\tensor 2}\bigr)$ with simple zeroes. The base $M$ then has a natural cotangent bundle structure, with $B$ the moduli space of curves of genus $g$, and $\rho\colon M\to B$ the obvious projection. Indeed, the tangent space to the moduli space of curves is $H^1(C,T_C)$, and Serre duality gives $H^0\bigl(C,\omega_C^{\tensor 2}\bigr)=H^1(C,T_C)^*$. Thus $M$ is an open subset of $T_B^*$ obtained by requiring that $Q_0$ has simple zeroes.

We note in passing that in the class $S[A_1]$ case the fibres of $\rho$ have a highly non-trivial compatibility with the Joyce structure: in the language of \cite[Section 8]{CH} they are good Lagrangians. See also the closely related notion of a projectable hyper-Lagrangian foliation from~\cite{DM2}. In the meromorphic case $l>0$, there is a similar picture, but the space $B$ is a moduli space of wild curves: the map $\rho$ remembers the curve $C$ together with the most singular half of the meromorphic tail of the differential $Q_0$ at each pole $x_i\in C$.

\subsection[Cluster-type structure on Z\_1]{Cluster-type structure on $\boldsymbol{Z_1}$}
\label{clust}

A crucial part of the definition of the $\tau$-function in Section \ref{tausection} will be the existence of collections of \emph{preferred co-ordinate systems} $(x_1,\dots,x_n)$ on the twistor fibre $Z_1$. These should be Darboux in the sense that
\[
\Omega_1 = \half \sum_{i,j} \omega_{ij} \cdot {\rm d}x_i\wedge {\rm d}x_j.\]

Although it will not be needed below, the preferred co-ordinates $(x_1,\dots,x_n)$ are expected to have the following asymptotic property. Consider an integral linear co-ordinate system on $M$ and the associated co-ordinate system on $X$ and take a point $x\in X$ with co-ordinates $(z_i,\theta_j)$. For $\epsilon \in \bC^*$, consider the point $\epsilon^{-1}\cdot x\in X$ with co-ordinates $\bigl(\epsilon^{-1} z_i,\theta_j\bigr)$ obtained by acting by~${\epsilon^{-1}\in \bC^*}$ on the point $x\in X$. The claim is that provided $x\in X$ is suitably generic there should be a preferred system of co-ordinates $(x_1,\dots,x_n)$ on $Z_1$ such that
 \begin{equation}
\label{asym}
x_i\bigl(\epsilon^{-1}\cdot x\bigr)\sim -\epsilon^{-1}z_i+\theta_i
\end{equation}
as $\epsilon\to 0$ in the half-plane $\Re(\epsilon)>0$.

When the Joyce structure is constructed via the DT RH problems of \cite{RHDT2}, the required preferred co-ordinates systems $x_i$ are given directly by the solutions to the RH problems, and the property \eqref{asym} holds by definition. Thus for the Joyce structures of interest in DT theory there is no problem finding this additional data. It is still interesting however to ask whether such distinguished co-ordinate systems exist for general Joyce structures. A heuristic explanation for why this might be the case can be found in \cite[Section 4.3]{BJoy}.

Consider the conjectural Joyce structures of class $S[A_1]$ of Section \ref{geometric}. Associated to the data~${g\geq 0}$ and
$m=\{m_1,\dots,m_l\}$ is a topological object $(\bS,\bM)$ called a marked bordered surface. It is a compact, connected, oriented surface with boundary $\bS$, together with a finite set of marked points $\bM\subset \partial \bS$. We specify the pair $(\bS,\bM)$ up to diffeomorphism by requiring that $\bS$ has genus~$g$, and $l$ boundary components, and that the numbers of marked points on the various boundary components are the integers $m_i-2\geq 0$.

Let us fix $\epsilon\in \bC^*$. Ignoring some subtleties \cite{AB}, the generalised monodromy of the projective structure \eqref{oper} then defines a $\PGL_2(\bC)$ framed local system on $(\bS,\bM)$ in the sense of Fock and Goncharov \cite{FG}.
By the definition of the Joyce structure in Conjecture \ref{secondguess}, this generalised monodromy is constant along the leaves of the foliation defined by $h_\epsilon$. We therefore expect an {\'e}tale map
\begin{equation}
 \label{stu}
 \mu_\epsilon\colon\ Z_\epsilon\to \cX(g,m)/\MCG(g,m),
\end{equation}
where $\cX(g,m)$ is the space of framed local systems on the surface $(\bS,\bM)$, and $\MCG(g,m)$ denotes the mapping class group.

When $l>0$ the preferred co-ordinate systems are expected to be the logarithms of Fock--Goncharov co-ordinates \cite{FG,GMN2}, and the asymptotic property \eqref{asym} should be a consequence of exact WKB analysis.
In the case $g=0$ and $m=\{7\}$, this story is treated in detail in~\cite{A2}.
In the case $l=0$ of holomorphic quadratic differentials, the expected picture is less well understood, but it has been suggested that the required co-ordinates on the character variety are the Bonahon--Thurston shear-bend co-ordinates \cite{Bon}.

\subsection[Lagrangian submanifolds of Z\_infty]{Lagrangian submanifolds of $\boldsymbol{ Z_\infty}$}
\label{cat}

At various points in what follows it will be convenient to choose a $\bC^*$-invariant Lagrangian submanifold $R\subset Z_\infty$ in the twistor fibre at infinity. For example, in the next section we will see that
the choice of such a Lagrangian, together with a cotangent bundle
structure on $Z_0$, leads to a time-dependent Hamiltonian system. Unfortunately, for the Joyce structures arising from the DT RH problems of~\cite{RHDT2}, the geometry of the twistor fibre $Z_\infty$ is currently quite mysterious, since it relates to the behaviour of solutions to the RH problems at $\epsilon=\infty$. So it is not yet clear whether such a Lagrangian submanifold can be expected to exist in general.

When discussing the twistor fibre $Z_\infty$ for Joyce structures of class $S[A_1]$, it is important to distinguish the cases $l=0$ and $l>0$. When $l=0$ the argument leading to the {\'e}tale map~\eqref{stu} applies also when $\epsilon=\infty$, because the monodromy of the projective structure \eqref{oper} lives in the same space for all $\epsilon^{-1}\in \bC$.
In contrast, in the case $l>0$, the projective structure \eqref{oper} has poles of order $m_i$ at the points $x_i\in C$ for all $\epsilon \in \bC^*$, but when $\epsilon=\infty$ these poles are of order~${\bigl\lfloor \half(m_i+1)\bigr\rfloor}$. This means that the direct analogue of the map \eqref{stu} for $\epsilon=\infty$ takes values in a space of framed local systems of a strictly lower dimension, and the fibres of this map are not well understood at present. A closely related problem is to understand the behaviour of the Fock--Goncharov co-ordinates of the projective structure \eqref{oper} in the limit as $\epsilon\to \infty$.

Another relevant difference between the holomorphic and meromorphic cases can be seen in the action of $\bC^*$ on $Z_\infty$. The action of $\bC^*$ on $X$ corresponds to fixing the curve $C$ and rescaling the differentials $(Q_0,Q_1,Q_2)$ with weights $(2,1,0)$ respectively. In the holomorphic case $l=0$ the twistor fibre $Z_\infty$ is locally parameterising the pair $(C,Q_2)$ and the $\bC^*$-action is therefore trivial. As in the previous paragraph, in the meromorphic case $l>0$ the twistor fibre $Z_\infty$ is not so easily described, and the induced action on $Z_\infty$ is in general non-trivial, as can be seen explicitly in the example of Section~\ref{a2}.

\subsection{Hamiltonian systems}
\label{hamintro}

A \emph{time-dependent Hamiltonian system} consists of the following data:
\begin{itemize}\itemsep=0pt
\item[(i)] a submersion $f\colon Y\to B$ with a relative symplectic form \smash{$\Omega\in H^0\bigl(Y,\wedge^2 T^*_{Y/B}\bigr)$},
\item [(ii)] a flat, symplectic connection $k$ on $f$,
\item[(iii)] a section $\varpi\in H^0(Y, f^*(T_B^*))$.
\end{itemize}

For each vector field $u\in H^0(B,T_B)$ there is an associated function
 \[
 H_u=(f^*(u),\varpi)\colon\ Y\to \bC.\]
There is then a pencil $k_\epsilon$ of symplectic connections on $f$ defined by
\[
k_\epsilon(u)=k(u)+\epsilon^{-1} \cdot \Omega^\sharp({\rm d}H_u).
\]
The system is called \emph{strongly-integrable} if these connections are all flat.

These definitions become more familiar when expressed in local co-ordinates. If we take co-ordinates $t_i$ on the base $B$, which we can think of as times, and $k$-flat Darboux co-ordinates~${(q_i,p_i)}$ on the fibres of $f$, we can write $\varpi=\sum_i H_i  {\rm d}t_i$ and view the functions $H_i\colon Y\to \bC$ as time-dependent Hamiltonians. The connection $k_\epsilon$ is then given by the flows
\begin{equation}\label{noteat}
k_\epsilon\left(\frac{\partial}{\partial t_i}\right)=\frac{\partial}{\partial t_i}+\frac{1}{\epsilon} \cdot \sum_j\left(\frac{\partial H_i}{\partial p_j}\frac{\partial }{\partial q_j}-\frac{\partial H_i}{\partial q_j}\frac{\partial }{\partial p_j}\right).
\end{equation}
The condition that the system is strongly-integrable is that
\begin{equation}\label{wan}
\sum_{r,s}\left(\frac{\partial H_i}{\partial q_r}\cdot \frac{\partial H_j}{\partial p_s}-\frac{\partial H_i}{\partial q_s}\cdot \frac{\partial H_j}{\partial p_r}\right)=0, \qquad \frac{\partial H_i}{\partial t_j}=\frac{\partial H_j}{\partial t_i}.
\end{equation}

Let $L\subset Y$ denote a leaf of the foliation $k_1$. The restriction $\varpi|_L$ is then closed, so we can write $\varpi|_L={\rm d}\log(\tau_L)$ for some locally-defined function $\tau_L\colon L\to \bC^*$. In terms of co-ordinates, since the projection $\pi\colon L\to B$ is a local isomorphism, we can lift $t_i$ to co-ordinates on $L$, whence we have
\begin{equation}\label{row}\frac{\partial}{\partial t_i} \log(\tau_L) =H_i|_L.\end{equation}
A $\tau$-function in this context is a locally-defined function $\tau\colon Y\to \bC^*$ whose restriction $\tau|_L$ to each leaf $L\subset Y$ satisfies \eqref{row}. Note that this definition only specifies $\tau$ up to multiplication by the pullback of an arbitrary function on the space of leaves.

\subsection{Hamiltonian systems from Joyce structures}
\label{hamham}

Let $M$ be a complex manifold equipped with a Joyce structure. Suppose
 also given
\begin{itemize}\itemsep=0pt
\item[(i)] a cotangent bundle structure $M\subset T^*_B$,

\item[(ii)] Lagrangian submanifold $R\subset Z_\infty$.
\end{itemize}
For the notion of a cotangent bundle structure, see Section \ref{cot}. We denote by $\rho\colon M\to B$ the induced projection, and by $\beta\in H^0(M,\rho^* (T_B^*))$ the tautological section. Note that $\beta$ is almost the same as the Liouville form $\lambda\in H^0(M,T_M^*)$ appearing in Section \ref{cot}: they correspond under the inclusion $\rho^*(T_B^*)\hookrightarrow T_M^*$ induced by $\rho$.

Set $Y=q_\infty^{-1}(R)\subset X$, and denote by $i\colon Y\hookrightarrow X$ the inclusion. There are maps
\[
\xymatrix@C=1em{ Y \ar@{^{(}->}[rr]^{i} && X \ar[rr]^{\pi} && M \ar[rr]^{\rho} && B.} \]
Define $p\colon Y\to M$ and $f\colon Y\to B$ as the composites $p=\pi\circ i$ and $f=\rho\circ \pi\circ i$.
We make the following transversality assumption:
\begin{itemize}\itemsep=0pt
\item[($\star$)] For each $b\in B$ the restriction of $q_1\colon X\to Z_1$ to the fibre $f^{-1}(b)\subset Y\subset X$ is {\'e}tale.\end{itemize}

The following result was proved in \cite{BJoy}.

\begin{Theorem}
\label{inta}
Given the above data there is a strongly-integrable time-dependent Hamiltonian system on the map $f\colon Y\to B$ uniquely specified by the following conditions:
\begin{itemize}\itemsep=0pt
 \item[$(i)$] the relative symplectic form $\Omega$ is induced by the closed $2$-form $ i^*(2{\rm i}\Omega_I)$ on $Y$;
 \item[$(ii)$] for each $\epsilon\in \bC^*$ the connection $k_\epsilon$ on $f\colon Y\to B$ satisfies
 \[
 \im (k_\epsilon)=T_Y\cap \im (h_\epsilon)\subset T_X;
 \]
 \item[$(iii)$] the Hamiltonian form is $\varpi=p^*(\beta)\in H^0(Y,f^*(T_B^*))$. \end{itemize}
\end{Theorem}

To make condition (iii) more explicit, take local co-ordinates $(t_1,\dots, t_d)$ on $B$, and extend them to co-ordinates $(s_i,t_j)$ on $M$ as in Section \ref{cot}. We can also extend to local co-ordinates~${(t_i,q_j,p_k)}$ on $Y$ as in Section \ref{hamintro}. Then $\beta=\sum_i s_i \cdot \rho^*({\rm d}t_i)$ and the connection $k_\epsilon$ is given by the flows \eqref{noteat} with Hamiltonians $H_i=p^*(s_i)$.
The main non-trivial claim is that the conditions~\eqref{wan} hold for these Hamiltonians.

\section[The Joyce structure tau-function]{The Joyce structure $\boldsymbol{\tau}$-function}
\label{tausection}

In this section, we define the $\tau$-function associated to a Joyce structure on a complex manifold~$M$, and discuss some of its basic properties. It is most naturally viewed as the unique up-to-scale local flat section of a flat line bundle on $X=T_M$. Given a choice of section of this line bundle it becomes a locally-defined function on $X$.

\subsection[Definition of the tau-function]{Definition of the $\boldsymbol{\tau}$-function}

Let $M$ be a complex manifold equipped with a Joyce structure as above, and let $p\colon Z\to \bP^1$ be the associated twistor space. We set
$\Theta_I= i_E(\Omega_I)$, so that as in Section \ref{ab} we have ${\rm d}\Theta_I=\Omega_I$.
Recall the identity of closed 2-forms on $X$
 \begin{equation}
 \label{ide}q_1^*(\Omega_1)=q_0^*(\Omega_0)+2{\rm i} \Omega_I + q_\infty^*(\Omega_\infty).\end{equation}
We start by giving the definition of the $\tau$-function in explicit local form. The geometrically-minded reader is encouraged to read the next section first.

\begin{Definition}
\label{deftau}
Choose locally-defined symplectic potentials:
\begin{itemize}\itemsep=0pt
\item[(i)] $\Theta_0$ on $Z_0$ satisfying ${\rm d}\Theta_0=\Omega_0$ and $\cL_E(\Theta_0)=2\Theta_0$,
\item[(ii)] $\Theta_1$ on $Z_1$ satisfying ${\rm d}\Theta_1=\Omega_1$,
\item[(iii)] $\Theta_\infty$ on $Z_\infty$ satisfying ${\rm d}\Theta_\infty=\Omega_\infty$ and $\cL_E(\Theta_\infty)=0$.
\end{itemize}
Then the corresponding \emph{$\tau$-function} is the locally-defined function on $X$ uniquely specified up to multiplication by constants by the relation
\begin{equation}
\label{tau}{\rm d}\log(\tau)=q_0^*(\Theta_0)+2{\rm i}\Theta_I +q_\infty^*(\Theta_\infty)-q_1^*(\Theta_1).\end{equation}
\end{Definition}

In the case of a homogeneous Joyce structure, we can pull back $\tau$ via the action
map $m\colon \bC^*\times X \to X$ so that it becomes a function also of $\epsilon\in \bC^*$. Restricted to the slice $\{\epsilon\}\times X$ it then satisfies
\begin{equation}
 \label{tau2}
 {\rm d}\log(\tau) = \epsilon^{-2}q_0^*(\Theta_0)+2{\rm i}\epsilon^{-1}\Theta_I +q_\infty^*(\Theta_\infty)-q_\epsilon^*(\Theta_\epsilon),
 \end{equation}
 where the symplectic potential $\Theta_\epsilon$ is defined by the relation $q_\epsilon^*(\Theta_\epsilon)=m_{\epsilon^{-1}}^* (q_1^*(\Theta_1))$, and we used the relation $\cL_E(\Theta_I)=\Theta_I$. Although the extra parameter $\epsilon$ is redundant, it is frequently useful to introduce it, for example, so as to expand $\tau$ as an asymptotic series.

 To obtain a well-defined $\tau$-function, we of course need to give some prescription for choosing the symplectic potentials $\Theta_0$, $\Theta_1$ and $\Theta_\infty$. We make some general remarks about this in Section~\ref{prims}, although
some aspects are still unclear. In particular, the choice of $\Theta_\infty$ remains quite mysterious. In later sections, we will show how to make appropriate choices in specific examples.

\subsection{Global description}
Let us assume that the de Rham cohomology class of the form $\frac{1}{2\pi {\rm i}}\cdot \Omega_1$ on the twistor fibre~$Z_1$ is integral. As explained in Appendix \ref{app}, this implies that there is a line bundle with connection~${(L_1,\nabla_1)}$ on $Z_1$ with curvature form $\Omega_1$. Note that the forms $\Omega_0$ and $2{\rm i}\Omega_I$ are exact by~\eqref{scales2}, so the relation \eqref{qe} shows that the same integrality condition holds for all the twistor fibres~$Z_\epsilon$. In particular, there are line bundles with connection $(L_0,\nabla_0)$ and $(L_\infty,\nabla_\infty)$ on the twistor fibres~$Z_0$ and $Z_\infty$, with curvature forms $\Omega_0$ and $\Omega_\infty$, respectively.

The identity \eqref{ide} shows that the connection
\[
\big(q_0^*(\nabla_0)\tensor 1\tensor 1\big) + \big(1\tensor q_\infty^*(\nabla_\infty)\tensor 1\big)- \big(1\tensor 1\tensor q_1^*(\nabla_1)\big)
+2{\rm i} \Theta_I\]
on the line bundle $q_0^*(L_0)\tensor q_\infty^*(L_\infty) \tensor q_1^*( L_1)^{-1}$
is flat. Locally on $X$ there is therefore a unique flat section up to scale, which we call the $\tau$-section.

Given a section $s_0\in H^0(U,L_0)$ over an open subset $U\subset Z_0$ we can write $\nabla_0(s_0)=\Theta_0\cdot s_0$ for a 1-form $\Theta_0$ on $U$. Then ${\rm d}\Theta_0=\Omega_0|_U$, so that $\Theta_0$ is a symplectic potential for $\Omega_0$ on this open subset. This construction defines a bijection between local sections of $L_0$ up to scale, and local symplectic potentials for $\Omega_0$.
Let us take local sections $s_0$, $s_1$ and $s_\infty$ of the bundles~$L_0$,~$L_1$ and~$L_\infty$ respectively, and let $\Theta_0$, $\Theta_1$ and $\Theta_\infty$ be the corresponding symplectic potentials. Then we can write the $\tau$-section in the form $\tau\cdot \bigl(s_0\tensor s_\infty\tensor s_1^{-1}\bigr)$, and the resulting locally-defined function~$\tau$ on $X$ will satisfy the equation \eqref{tau}.

We can constrain the possible choices of local sections $s_0$ and $s_\infty$ using the $\bC^*$-actions on the fibres $Z_0$ and $Z_\infty$. The relations $\cL_E(\Omega_0)=2\Omega_0$ and $\cL_E(\Omega_\infty)=0$ show that the symplectic forms $\Omega_0$ and $\Omega_\infty$ are homogeneous for this action. It follows that the action lifts to the pre-quantum line bundles $L_0$ and $L_\infty$. We can then insist that the local sections $s_0$ and $s_\infty$ are $\bC^*$-equivariant, which in terms of the corresponding symplectic potentials translates into the conditions $\cL_E(\Theta_0)=2\Theta_0$ and $\cL_E(\Theta_\infty)=0$ appearing in Definition \ref{deftau}.

\begin{Remark}
For Joyce structures arising in DT theory the integrality property of the form $\frac{1}{2\pi {\rm i}}\cdot \Omega_1$ amounts to an extension of the usual wall-crossing formula.
In brief, the twistor fibre~$Z_1$ is covered by preferred Darboux co-ordinate charts whose transition functions are symplectic maps given by time 1 Hamiltonian flows of logarithms of products of quantum dilogarithms. The cocycle condition is the wall-crossing formula in DT theory. The generating functions for these symplectic maps are given by formulae involving the Rogers dilogarithm, and the cocycle condition for the pre-quantum line bundle $(L_1,\nabla_1)$ is then an extension of the wall-crossing formula involving these generating functions.
This picture was explained by Alexandrov, Persson and Pioline \cite{AMPP,APP}, and by Neitzke \cite{Nei}, and in the case of theories of class $S[A_1]$ plays an important role in the work of Teschner et al.\ \cite{T1,T2}. The Rogers dilogarithm identities were described in the context of cluster theory by Fock and Goncharov \cite[Section 6]{FG2}, and by Kashaev and Nakanishi \cite{KN,Nak}.
\end{Remark}

 \subsection{Choice of symplectic potentials}
\label{prims}

It was explained in Section \ref{ab} that setting
 \[
 \Theta_0=\half  i_E(\Omega_0)= \half \sum_{i,j} \omega_{ij} z_i  {\rm d}z_j,
 \]
provides a canonical and global choice for $\Theta_0$.
On the other hand, the key to defining $\Theta_1$ is to assume the existence of a distinguished Darboux co-ordinate system $(x_1,\dots,x_n)$ on the twistor fibre $Z_1$ as in Section \ref{clust}.
We can then take
 \begin{equation}\label{cano}\Theta_1 = \half \sum_{i,j} \omega_{ij} x_i  {\rm d}x_j.\end{equation}

The choice of $\Theta_\infty$ remains quite mysterious in general. One way to side-step this problem is to choose a $\bC^*$-invariant Lagrangian $R\subset Z_\infty$ as in Section \ref{cat}, since if we restrict $\tau$ to the inverse image $Y=q_\infty^{-1}(R)$ we can then drop the term $q_\infty^*(\Theta_\infty)$ from the definition of the $\tau$-function. It is not quite clear whether this procedure is any more than a convenient trick. In other words, it is not clear whether $\tau$ should be viewed as a function on $X$, depending on a choice of symplectic potential $\Theta_+$, or whether the natural objects are the Lagrangian $R\subset Z_\infty$, and a function $\tau$ defined on the corresponding submanifold $Y\subset X$.

In practice, in examples, the above choices of symplectic potentials $\Theta_0$ and $\Theta_1$
do not always give the nicest results. We discuss several possible modifications here. Of course, from the global point-of-view, these variations correspond to expressing the same $\tau$-section in terms of different local sections of the line bundles $L_0$ and $L_1$.

\subsubsection{Cotangent bundle}
\label{three}

Suppose we are given a cotangent bundle structure on $M$ as discussed in Section \ref{cot}. As explained there, it is natural to assume that the associated Liouville form $\lambda$ satisfies $\cL_E(\lambda)=2\lambda$.
We can then consider the following three choices of symplectic potential
\[
\Theta^{\rm L}_0=-\lambda, \qquad \Theta_0=\half  i_E(\Omega_0), \qquad \Theta^{\rm H}_0= i_E(\Omega_0)+\lambda.\]
The justification for the strange-looking primitive $\Theta^{\rm H}_0$ will be explained in Section \ref{spam}: it is the correct choice to produce $\tau$-functions for the Hamiltonian systems of Section \ref{hamham}.
Applying the Cartan formula gives
\[
\Theta^{\rm H}_0-\Theta_0=\Theta_0-\Theta^{\rm L}=\half  i_E(\Omega_0)+\lambda=-\half  i_E({\rm d}\lambda)+\half \cL_E(\lambda)=\half {\rm d}  i_E(\lambda).
\]
Thus the resulting $\tau$ functions differ by the addition of the global function $\half q_0^*( i_E(\lambda))$.

\subsubsection{Polarisation}
\label{polar}
Suppose the distinguished Darboux co-ordinate system $(x_1,\dots,x_{n})$ on $Z_1$ is polarised, in the sense that the coefficients $\omega_{ij}$ appearing in \eqref{cano} satisfy
$\omega_{ij}=0$ unless $|j-i|=d$, where $n=2d$.
 We can then take as symplectic potential on $Z_1$
 \[
 \Theta^{\rm P}_1=\sum_{i=1}^d \omega_{i,i+d} \cdot x_i  {\rm d}x_{i+d},
 \]
This resulting $\tau$-function will be modified by
$\half\sum_i\omega_{i,i+d} \cdot x_i x_{i+d} $.

\subsubsection[Flipping Theta\_I]{Flipping $\boldsymbol{\Theta_I}$}
\label{flip}
Given local co-ordinates on $X$ as in Section \ref{coords} there are in fact two obvious choices of symplectic potential for $2{\rm i}\Omega_I$, namely
\[
2{\rm i}\Theta_I=-\sum_{p,q} \omega_{pq} z_p  {\rm d} \theta_q, \qquad 2{\rm i}\Theta'_I=\sum_{p,q} \omega_{pq} \theta_p  {\rm d} z_q.
\]
In Section \ref{flipping}, it will be convenient to replace $2{\rm i}\Theta_I$ in the definition of the $\tau$-function with~$2{\rm i}\Theta'_I$. This
will change the $\tau$-function by $K=\sum_{p,q} \omega_{pq} \cdot z_p\theta_q$. Note that $K\colon X\to \bC$ is a globally-defined function, since it is the 1-form $ i_E (\omega)$ on $M$ considered as a function on $X=T_M$. It does not however descend to the quotient $X^\hash=T_M/(2\pi {\rm i}) T^{\bZ}_M$.

\subsection[Interpretations of the tau-function]{Interpretations of the $\boldsymbol{\tau}$-function}

By restricting to various submanifolds of $X$, the $\tau$-function can be viewed as a generating function in a confusing number of ways.

\subsubsection{Restriction to the zero-section}
\label{restr}

Let $j\colon M\hookrightarrow X=T_M$ be the inclusion of the zero-section, defined by setting all co-ordinates $\theta_i=0$. Since this is the fibre $q_\infty^{-1}(0)$ over the distinguished point $0\in Z_\infty$ of \cite[Section 5.2]{BJoy}, we have $j^*q_\infty^*(\Theta_\infty)=0$. The formula \eqref{snowy} for $\alpha_I=\Theta_I$ shows that also $j^*(\Theta_I)=0$. The defining relation of the $\tau$-function then implies that
\[
{\rm d}\log( \tau|_M) =\Theta_0 -j^*q_1^*(\Theta_1),\]
 so that $\log(\tau|_M)$ is the generating function for the symplectic map $q_1\circ j\colon M\to Z_1$ with respect to the symplectic potentials $\Theta_0$ and $\Theta_1$.

\subsubsection[Restriction to a fibre of q\_0]{Restriction to a fibre of $\boldsymbol{ q_0}$}
\label{flipping}
Let $j\colon F= T_{M,p}\hookrightarrow X$ be the inclusion of a fibre of the projection $\pi\colon X=T_M\to M$. Restriction to this locus corresponds to fixing the co-ordinates $z_i$. By \eqref{w}, the restriction
\[
\Omega_F=j^* q_\infty^*(\Omega_\infty)= \half\sum_{p,q}\omega_{pq} \cdot {\rm d}\theta_p\wedge {\rm d}\theta_q
\]
is the linear symplectic form on $T_{M,p}$ defined by the symplectic form $\omega$. It follows that $\Theta_F=j^* q_\infty^*(\Theta_\infty)$ is a symplectic potential for this form.

Let us take the flipped form $2{\rm i}\Theta'_I$ in the definition of the $\tau$-function as in Section \ref{flip}.
The forms $\Theta_0$ and $2{\rm i}\Theta'_I$ vanish when restricted to $F$, so
\[
{\rm d}\log(\tau|_F)=j^*q_\infty^*(\Theta_\infty) -j^*q_1^*(\Theta_1),
\]
and $\log(\tau|_F)$ is the generating function for the symplectic map $q_1\circ j\colon F\to Z_1$ with respect to the symplectic potentials $\Theta_F$ and $\Theta_1$.

\subsubsection[Restriction to a fibre of f]{Restriction to a fibre of $\boldsymbol{f}$}

Let us consider the setting of Section \ref{hamham}. Thus we have chosen a cotangent bundle structure~${M\subset T_B^*}$, with associated projection $\rho\colon M\to B$, and a $\bC^*$-invariant Lagrangian submanifold $R\subset Z_\infty$, and set $Y=q_\infty^{-1}(R)$. Let $j\colon F\into Y$ be the inclusion of a fibre $F=f^{-1}(b)$ of the map $f\colon Y\to B$. By Theorem \ref{inta}\,(i), the closed 2-form $2{\rm i}\Omega_I$ restricts to a symplectic form on $F$. As explained above, we can drop the term $q_\infty^*(\Theta_\infty)$ from the definition of the $\tau$-function after restricting to $Y$. Let us take the symplectic potential on $Z_0$ to be $\Theta_0^{\rm L}=-d\lambda$ as in Section~\ref{three}. Then since $\pi(F)\subset \rho^{-1}(b)$ we also have $j^* q_0^*\bigl(\Theta_0^{\rm L}\bigr)=0$. Thus
\[
{\rm d}\log(\tau|_F)=j^*(2{\rm i}\Theta_I) -j^*q_1^*(\Theta_1),\]
and $\log(\tau|_F)$ is the generating function for the symplectic map $q_1\circ j\colon F\to Z_1$ with respect to the symplectic potentials $j^*(2{\rm i}\Theta_I)$ and $\Theta_1$.

\subsubsection[Hamiltonian system tau-function]{Hamiltonian system $\boldsymbol{\tau}$-function}
\label{spam}

Consider again the setting of Section \ref{hamham}. Let us further assume that the Joyce structure is homogenous and that the Lagrangian $R\subset Z_\infty$ is preserved by the $\bC^*$-action. As before, after restriction to $Y\subset X$ we can drop the term $q_\infty^*(\Theta_\infty)$ from the definition of the $\tau$-function. Let us take the symplectic potential on $Z_0$ to be $\Theta_0^{\rm H}$ as defined in Section \ref{three}. Then we have
\begin{align*}
{\rm d}\log(\tau|_Y)&=i^*q_0^*(\lambda) + i^*q_0^*(  i_E(\Omega_0))+ i^*( i_E(2{\rm i}\Omega_I))-i^*q_1^*(\Theta_1)\\
&=p^*(\lambda)+i^*( i_E(q_1^*(\Omega_1)))-i^* q_1^*(\Theta_1).
\end{align*}
Here we used the identity \eqref{ide}, together with the assumption that the Lagrangian $R\subset Z_\infty$ is $\bC^*$-invariant, which ensures that $ i_E(\Omega_\infty)|_R=0$.
 Let $L\subset X$ be a leaf of the connection~$k_1$ on~${f\colon Y\to B}$. By construction of $k_1$, this is the intersection of $Y$ with a leaf of the connection $h_1$ on $\pi\colon X\to M$. Note that if $u$ is a horizontal vector field for the connection~$h_1$ then~${i_u  i_E (q_1^*(\Omega_1))=- i_E i_u (q_1^*(\Omega_1))=0}$. Thus
\[{\rm d}\log(\tau|_L)= p^*(\lambda)|_L.\]
Applying the definition of Section \ref{hamintro}, it follows that $\tau|_Y$ is a $\tau$-function for the strongly-integrable time-dependent Hamiltonian system of Theorem \ref{inta}.

\section{Example: uncoupled BPS structures}
\label{uncoupled}

In the paper \cite{RHDT1}, it was found that in certain special cases, solutions to DT RH problems could be encoded by a single generating function, which was denoted $\tau$. The most basic case is the one arising from the DT theory of the doubled A$_1$ quiver, where the resulting $\tau$-function is a~modified Barnes $G$-function \cite[Section 5]{RHDT1}. In the case of the DT theory of the resolved conifold~$\tau$ was shown to be a variant of the Barnes triple sine function \cite{con}, and interpreted as a non-perturbative topological string partition function. In this section, we show that these $\tau$-functions can be viewed as special cases of the more general definition given above.

\subsection[Uncoupled BPS structures and associated tau-function]{Uncoupled BPS structures and associated $\boldsymbol{\tau}$-function}
Consider as in \cite[Section 5.4]{RHDT1} a framed, miniversal family of finite, integral BPS structures over a complex manifold $M$. At each point $p\in M$ there is a BPS
structure consisting of a~fixed lattice~${\Gamma\isom \bZ^{\oplus n}}$, with a skew-symmetric form $\<-,-\>$, a central charge $Z_p\colon \Gamma\to \bC$, and a collection of BPS invariants $\Omega_p(\gamma)\in \bQ$ for $\gamma\in \Gamma$. The miniversal assumption is that the central charges~${z_i=Z(\gamma_i)}$ of a collection of basis vectors $\gamma_i\in \Gamma$ define local co-ordinates on~$M$. The finiteness assumption ensures that only finitely many $\Omega_p(\gamma)$ are nonzero for any given point~${p\in M}$, and the integrality condition is that $\Omega_p(\gamma)\in \bZ$ for a generic point $p\in M$.

Let us also assume that all the BPS structures parameterised by $M$ are uncoupled, which means that $\Omega_p(\gamma_i)\neq 0$ for $i=1,2$ implies $\<\gamma_1,\gamma_2\>=0$. This is a very special assumption, which implies \cite[Remark A.4]{RHDT1} that
 the BPS invariants $\Omega_p(\gamma)=\Omega(\gamma)$ are independent of $p\in M$. We can then take a basis $(\gamma_1,\dots, \gamma_{2d})$, where $n=2d$ as before, such that $\<\gamma_i,\gamma_j\>=0$ unless $|j-i|=d$, and such that $\Omega(\gamma)\neq 0$ implies that $\gamma\in \bigoplus_{i=1}^d \bZ \gamma_i$. We set $\eta_{ij}=2\pi {\rm i} \cdot \<\gamma_i,\gamma_j\>$ and take $\omega_{ij}$ to be the inverse matrix. Note that $\omega_{i,i+d}\cdot \eta_{i,i+d}=-1$ for all $1\leq i\leq d$.

 At each point $p\in M$ there is a DT RH problem depending on the BPS structure, and also on a twisted character $\xi\colon \Gamma\to \bC^*$. For $1\leq i\leq d$ the expressions $\exp\bigl(-\epsilon^{-1} z_i \bigr)\cdot \xi(\gamma_i)$ are solutions to this problem. We assume that $\xi(\gamma_i)=1$ for all $1\leq i \leq d$ which then implies that~${\Omega(\gamma)\neq 0 \implies \xi(\gamma)=1}$. This amounts to fixing a Lagrangian $R\subset Z_\infty$. Then by \cite[Theorem~5.3]{RHDT1}, the DT RH problem has a unique solution whose components $X_i=\exp(x_i)$ can be written in the form
 \[
 x_i=-\epsilon^{-1} z_i + y_i, \qquad y_i=\sum_{\gamma\in \Gamma} \Omega(\gamma)\cdot \<\gamma,\gamma_i\> \cdot \log \Lambda\bigg(\frac{Z(\gamma)}{2\pi {\rm i} \epsilon}\bigg),
 \]
 where $\Lambda(w)$ is the modified gamma function
 \[
\Lambda(w)=\frac{{\rm e}^{w} \cdot\Gamma(w)}{\sqrt{2\pi}\cdot w^{w-\half}}.\]

Note that for $1\leq j \leq d$ we have $y_i=0$, and
\[
2\pi {\rm i} \omega_{i,i+d} \cdot y_{i+d}= -\sum_{k_1,\dots,k_d\in \bZ } \Omega\left(\sum_{p=1}^d k_p \gamma_p \right)\cdot k_i \cdot \log \Lambda\left((2\pi {\rm i} \epsilon)^{-1} \sum_{p=1}^d k_p z_p \right).\]
This implies the relations
\begin{equation}
\label{rrr}\omega_{i,i+d} \cdot \frac{\partial y_{i+d}}{\partial z_j}= \omega_{j,j+d} \cdot \frac{\partial y_{j+d}}{\partial z_i}, \qquad \sum_{j=1}^d z_j \cdot \frac{\partial y_{i+d}}{\partial z_j} +\epsilon \cdot \frac{\partial y_{i+d}}{\partial \epsilon}=0.\end{equation}
The $\tau$-function of \cite{RHDT1} was then defined as a locally-defined function $\tau\colon M\to \bC^*$ satisfying
\begin{equation}
\label{tausnow}\frac{\partial}{\partial z_{i}} \log(\tau)= -\omega_{i,i+d}\cdot \frac{\partial y_{i+d}}{\partial \epsilon},\qquad \frac{\partial}{\partial z_{i+d}} \log(\tau)=0
\end{equation}
for $1\leq i \leq d$, and homogeneous under simultaneous rescaling of all $z_i$ and $\epsilon$.

\subsection[Comparison of tau-functions]{Comparison of $\boldsymbol{\tau}$-functions}
We can give $M$ a cotangent bundle structure in which the map $\rho\colon M\to B$ just projects to the
co-ordinates $(z_1,\dots,z_d)$. We then have
\[
\lambda=\sum_{i=1}^d \omega_{i,i+d} \cdot z_{i+d}  {\rm d}z_i, \qquad \Omega_0=\sum_{i=1}^d \omega_{i,i+d} \cdot {\rm d}z_i \wedge {\rm d}z_{i+d}.
\]
Let us take the Hamiltonian system choice $\Theta_0=\Theta_0^{\rm H}$ from Section \ref{three}, and the polarised choice~${\Theta_1=\Theta_1^{\rm P}}$ from Section \ref{polar}. Then
\[
\Theta_0= \sum_{i=1}^d \omega_{i,i+d} \cdot z_i  {\rm d}z_{i+d}, \qquad \Theta_1=\sum_{i=1}^d \omega_{i,i+d} \cdot x_i  {\rm d}x_{i+d}.
\]
 Let us also restrict to a section $M\subset X=T_M$ by fixing the constant term $\xi$. Since the variables~$\theta_i$ are then constant, the restriction of the form $\Theta_I$ is zero. The definition \eqref{tau2} of the $\tau$-function becomes
\[
{\rm d}\log(\tau|_M) =\epsilon^{-1} \cdot \sum_{i=1}^d \omega_{i,i+d} \cdot z_i {\rm d} y_{i+d}.\]

Expressing $\tau$ as a function of the co-ordinates $z_i$, we find that for $1\leq i\leq d$
\[
\frac{\partial}{\partial z_{i}} \log(\tau|_M)=\epsilon^{-1}\cdot \sum_{j=1}^d \omega_{j,j+d} \cdot z_j \frac{\partial y_{j+d}}{\partial z_i} , \qquad \frac{\partial}{\partial z_{i+d}} \log(\tau|_M)=0.
\]
Using the relations \eqref{rrr}, this gives
\[
 \frac{\partial}{\partial z_{i}} \log(\tau|_M)=\epsilon^{-1} \cdot \sum_{j=1}^d \omega_{i,i+d}\cdot z_j \frac{\partial y_{i+d}}{\partial z_j}= -\omega_{i,i+d}\cdot \frac{\partial y_{i+d}}{\partial \epsilon},
 \]
which coincides with \eqref{tausnow}.

 Thus we see that the $\tau$-functions of \cite{RHDT1} are particular examples of the $\tau$-functions introduced here. The non-perturbative topological string partition function for the resolved conifold obtained in \cite{con} also fits into this framework. Although the relevant BPS structures are not finite, they are uncoupled, and the above analysis goes through unchanged.

\section[Example: the A\_2 quiver, cubic oscillators and Painlev{\'e} I]{Example: the A$\boldsymbol{_2}$ quiver, cubic oscillators and Painlev{\'e} I}
\label{a2}

This example arises from the DT theory of the A$_2$ quiver, and is the particular example of the construction of Section \ref{geometric} corresponding to $g=0$ and $m=\{7\}$. It was studied in detail in \cite{A2} and \cite[Section 9]{BJoy} to which we refer the reader for further details. We show that with the natural choices of symplectic potentials $\Theta_0$, $\Theta_1$, $\Theta_\infty$, the resulting $\tau$-function coincides with the Painlev{\'e} I $\tau$-function considered by Lisovyy and Roussillon \cite{LR}.

\subsection{Joyce structure}
The base of the Joyce structure is
$M=\bigl\{(a,b)\in \bC^2\mid 4a^3+27b^2\neq 0\bigr\}$.
Associated to a pair $(a,b)\in M$ is a quadratic differential
 \begin{equation}
\label{coffee}Q_0(x) {\rm d}x^{\tensor 2}=\bigl(x^3+ax+b\bigr)  {\rm d}x^{\tensor 2}\end{equation}
 on $\bP^1$ which has a single pole of order 7 at $x=\infty$ and simple zeroes. The associated double cover is an affine elliptic curve
 \[
 \Sigma^0=\Sigma^0(a,b)=\bigl\{(x,y)\in \bC^2\mid y^2 =x^3+ax+b\bigr\}.
 \]

We consider the pencil of projective structures
\begin{equation}
\label{coffeee!}f''(x)=Q(x)\cdot f(x), \qquad Q(x)=\epsilon^{-2}\cdot Q_0(x) + \epsilon^{-1}\cdot Q_1(x) + Q_2(x),\end{equation}
where the terms in the potential are
\[
Q_1(x)=\frac{p}{x-q}+r, \qquad
Q_2(x)=\frac{3}{4(x-q)^2}+\frac{r}{2p(x-q)}+\frac{r^2}{4p^2},
\]
and we impose the relation
$
p^2=q^3+aq+b$,
which ensures that \eqref{coffee} has an apparent singularity at $x=q$ for all $\epsilon\in \bC^*$. Thus the pencil of projective structures \eqref{coffeee!} is parameterised by a point $(a,b)\in M$, a point $(q,p)\in \Sigma^0(a,b)$, and a point $r\in \bC$.

Locally, we can take alternative co-ordinates $(z_1,z_2,\theta_1,\theta_2)$ defined by the relations
\[
z_i=\int_{\gamma_i} y  {\rm d}x, \qquad \theta_i=-\int_{\gamma_i} \bigg(\frac{p}{x-q} + r\bigg) \frac{{\rm d}x}{2y},
\]
where $(\gamma_1,\gamma_2)\subset H_1(\Sigma,\bZ)$ is a basis of cycles with intersection $\gamma_1\cdot \gamma_2=1$.

The Joyce structure on $X=T_M$ is obtained by taking the isomonodromy connection $h_\epsilon$ for the above family of projective structures. Explicitly, it is given by
\begin{align*}
h_\epsilon\bigg(\frac{\partial}{\partial a}\bigg)&=-\frac{2p}{\epsilon} \frac{\partial}{\partial q}
-\frac{q}{\epsilon} \frac{\partial}{\partial r} +\bigg(\frac{\partial}{\partial a}-\frac{r}{p}\frac{\partial}{\partial q}
-\frac{r^2\bigl(3q^2+a\bigr)-qpr}{2p^3} \frac{\partial}{\partial r}\bigg),\\
h_\epsilon\bigg(\frac{\partial}{\partial b}\bigg)&=-\frac{1}{\epsilon}\frac{\partial}{\partial r}+\bigg(\frac{\partial}{\partial b}
 +\frac{r}{2p^2} \frac{\partial}{\partial r}\bigg).
 \end{align*}
 To obtain the Pleba{\'n}ski function $W$, one rewrites these flows in the co-ordinates $(z_i,\theta_j)$. The explicit formula can be found in \cite[Section 9]{BJoy}.

\subsection{Euler vector fields and symplectic forms}

 The Euler vector field on $M$ is
 \[
 Z=z_1\frac{\partial}{\partial z_1}+z_2\frac{\partial}{\partial z_2}= \frac{4a}{5} \frac{\partial}{\partial a} + \frac{6b}{5} \frac{\partial}{\partial b},
 \]
 and the lift to $X$ is
 \[
 E=z_1\frac{\partial}{\partial z_1}+z_2\frac{\partial}{\partial z_2}=\frac{4a}{5} \frac{\partial}{\partial a} + \frac{6b}{5} \frac{\partial}{\partial b}+ \frac{2q}{5} \frac{\partial}{\partial q}+ \frac{r}{5} \frac{\partial}{\partial r}.
 \]
The following facts were obtained in \cite[Section 9]{BJoy}.
There are identities
\begin{gather}
 \Omega_0=-\frac{1}{2\pi {\rm i}} {\rm d}z_1\wedge {\rm d}z_2 = {\rm d}a\wedge {\rm d}b,\nonumber \\
2{\rm i}\Omega_I=\frac{1}{2\pi {\rm i}} ({\rm d}\theta_1\wedge {\rm d}z_2 - {\rm d}\theta_2\wedge {\rm d}z_1)={\rm d}q\wedge {\rm d}p+{\rm d}a\wedge {\rm d}r, \label{flour}
\end{gather}
which then immediately give
\begin{gather*}
  i_E(\Omega_0)=-\frac{6b}{5} {\rm d}a+\frac{4a}{5} {\rm d}b, \qquad  i_E(2{\rm i}\Omega_I)=\frac{2q}{5} {\rm d}p-\frac{3p}{5} {\rm d}q +\frac{4a}{5} {\rm d}r-\frac{r}{5} {\rm d}a.\end{gather*}

Moreover, the functions
\[
\phi_1=q+\frac{ar}{p}, \qquad \phi_2=\frac{r}{2p}, \]
descend to local co-ordinates on the twistor fibre $Z_\infty$, and satisfy the relation
$\Omega_\infty={\rm d}\phi_1\wedge {\rm d}\phi_2$,
which can be used to give an explicit expression for $\Omega_\infty$.

\subsection[Joyce structure tau-function]{Joyce structure $\boldsymbol{\tau}$-function}
Let us now consider the additional choices needed to define the Joyce structure $\tau$-function. Firstly, there is a natural cotangent bundle structure on $M$ for which $\rho\colon M\to B$ is the projection to the co-ordinate $a$, which is the Painlev{\'e} time. The associated Liouville form is $\lambda = b {\rm d}a$.

We shall
restrict to the Lagrangian $R\subset Z_\infty$ defined by the equation $\phi_2=0$. The inverse image $Y=q_\infty^{-1}(L)\subset X$ is then the 3-dimensional locus $r=0$.
 We can take $(a,q,p)$ as co-ordinates on the space $Y$. The map $f\colon Y\to B$ of Section \ref{hamham} is then the projection $(a,q,p)\mapsto a$. By~\eqref{flour}, the form $2{\rm i}\Omega_I$ induces the relative symplectic form ${\rm d}q\wedge {\rm d}p$. The horizontal leaves of the connection $k_\infty$ are obtained by varying $a$ while keeping $(q,p)$ fixed. The Hamiltonian form is $\varpi=b {\rm d}a$.

The twistor fibre $Z_1$ is the space of framed local systems, and is covered by birational Fock--Goncharov co-ordinate charts $(\exp(x_1),\exp(x_2))$. We take the polarised choice for the symplectic potential $\Theta_1$, and the Hamiltonian system choice for $\Theta_0$. Thus
\[
\Theta_0= i_E(\Omega_0)+\lambda, \qquad \Theta_1=-(2\pi {\rm i})^{-1} x_1 {\rm d}x_2.
\]
Omitting the pullbacks $q^*$ from the notation, the definition of the $\tau$-function reads
\[
 {\rm d}\log(\tau) =\epsilon^{-2} \cdot \Theta_0 +\epsilon^{-1} \cdot 2{\rm i}\Theta_I +\Theta_\infty-\Theta_1.\]
With the above choices this gives
\begin{equation}
\label{my}
{\rm d}\log(\tau|_Y) =\epsilon^{-2}\left(-\frac{6b}{5} {\rm d}a +\frac{4a}{5} {\rm d}b+b{\rm d}a\right)+\epsilon^{-1} \left(\frac{2q}{5} {\rm d}p- \frac{3p}{5} {\rm d}q\right)+\frac{1}{2\pi {\rm i}} x_1 {\rm d}x_2.\end{equation}

We now compare with the Painlev{\'e} $\tau$-function computed in \cite{LR} and use notation as there. Consider the form $\omega$ defined in \cite[equation~(3.4)]{LR}. It involves local co-ordinates $t$, $m_a$, $m_b$ with $m_a$, $m_b$ constant under Painlev{\'e} flow. Let us re-express $\omega$ in terms of the local co-ordinates $t$, $q$, $p$. Using the relations $q_t=p$, $p_t=6q^2+t$ and $H_t=-q$ appearing on \cite[page 1]{LR}, we get
\begin{align*}
\frac{\omega}{2}&=H {\rm d}t +\frac{1}{5}\bigl(4t {\rm d}H + 3q_t {\rm d}q -2q {\rm d}q_t-\bigl(4t H_t+3q_t^2-2q q_{tt}\bigr){\rm d}t\bigr)\\&=-\frac{H}{5} {\rm d}t + \frac{4t}{5} {\rm d}H + \frac{3p}{5} {\rm d}q-\frac{2q}{5} {\rm d}p.
\end{align*}
It is shown in \cite{LR} that $\Omega={\rm d}\omega = 4\pi {\rm i} {\rm d}\nu_1 \wedge {\rm d}\nu_2$ and then the $\tau$-function $\tau_{\LR}$ is defined by
\begin{equation}
\label{lis} {\rm d}\log(\tau_{\LR})=\frac{1}{2}\omega- 2\pi {\rm i} \nu_1 {\rm d}\nu_2.\end{equation}

To compare with the $\tau$-function $\tau_{\TB}$ of the previous section, set $p_{\LR}=-2p_{\TB}$, $q_{\LR}=q_{\TB}$, $t=2a$, $H=2b$. Here the subscript $\TB$ means as in this paper, whereas $\LR$ means as appearing in \cite{LR}. Also we should set $r=0$ and $\epsilon=\half$. The projective structure \eqref{coffeee!} then becomes gauge equivalent to the system (2.1a) of \cite{LR}. To match the monodromy data, we set $x_1 = 2\pi {\rm i} \nu_1$ and $x_2=-2\pi {\rm i} \nu_2$. Then in terms of the notation of this paper, the definition \eqref{lis} becomes
\[
{\rm d}\log \tau_{\LR} = -\frac{4b}{5} {\rm d}a + \frac{16a}{5} {\rm d}b - \frac{6p}{5} {\rm d}q+\frac{4q}{5} {\rm d}p + \frac{1}{2\pi {\rm i}} x_1 {\rm d}x_2,
\]
which coincides with \eqref{my}. It follows that the two $\tau$-functions, which are both well defined up to multiplication by a nonzero constant, coincide.

\appendix
 \section{Pre-quantum line bundles in the holomorphic setting}
\label{app}

The following result on pre-quantum line bundles in the holomorphic setting is standard and well known, but is so relevant to the definition of the $\tau$-function that it seems worth briefly recalling the proof. As in the body of the paper, all line bundles, connections, symplectic forms etc., will be holomorphic, but to make clear the distinction from the more familiar geometric quantization story we will re-emphasize this at several places.

\begin{Theorem}
Let $M$ be a complex manifold equipped with a holomorphic symplectic form $\Omega$. Assume that the de Rham cohomology class $[\Omega]$ satisfies the integrality condition \[
\frac{1}{2\pi {\rm i}}\cdot [\Omega]\in H^2(M,\bZ)\subset H^2(M,\bC).
\]
Then there is a holomorphic line bundle with holomorphic connection whose curvature is $\Omega$.
\end{Theorem}

\begin{proof}
Consider the following diagram of sheaves of abelian groups on $M$, in which ${\rm d}\O$ is the sheaf of closed holomorphic 1-forms, ${\rm d}$ is the de Rham differential, and the unlabelled arrows are the obvious inclusions,
\begin{equation*}
\begin{gathered}
\xymatrix@C=1.8em{ \bZ \ar[d] \ar[rr]^{\cdot 2\pi {\rm i}} && \bC \ar[d]\ar[rr]^{\exp} && \bC^*\ar[d]\\
\bZ \ar[rr]^{\cdot 2\pi {\rm i}} && \O \ar[rr]^{\exp} \ar[d]_{d}&& \O^*\ar[d]_{{\rm d}\log} \\
&& {\rm d}\O \ar[rr] && {\rm d}\O.
}\end{gathered}
 \end{equation*}

The sheaf of 1-forms satisfying ${\rm d}\Theta=\Omega$ is a torsor for the sheaf ${\rm d}\O$, and hence defines an element $\eta\in H^1(M,{\rm d}\O)$. The image of $\eta$ via the boundary map in the central column is the class~${[\Omega]\in H^2(M,\bC)}$. By the integrality assumption, the image of $\eta$ via the boundary map in the right-hand column is $1\in H^2(M,\bC^*)$. Thus, by the long exact sequence in cohomology for the right-hand column there exist elements $\phi\in H^1(M,\O^*)$ satisfying ${\rm d}\log(\phi)=\eta$. Such a~class~$\phi$ defines a line bundle $L$ on $M$, and one can then see that $L$ has a holomorphic connection with curvature $\Omega$.

Translating the above discussion into {\v C}ech cohomology gives the following. Take a covering of $X$ by open subsets $U_i$ such that all intersections $U_{i_1,\dots, i_n}=U_{i_1}\cap \cdots\cap U_{i_n}$ are contractible. Choose 1-forms $\Theta_i$ on $U_i$ satisfying ${\rm d}\Theta_i=\Omega|_{U_i}
$, and set $\Theta_{ij}=\Theta_i|_{U_{ij}}-\Theta_j|_{U_{ij}}$. Then ${\rm d}\Theta_{ij}=0$ and the collection $\{\Theta_{ij}\}$ defines a \v{C}ech 1-cocycle for the sheaf ${\rm d}\O$. On $U_{ij}$ we can now write
 \begin{equation}
\label{dull}{\rm d}\log \phi_{ij}=\Theta_{ij}=\Theta_i|_{U_{ij}}-\Theta_j|_{U_{ij}}
\end{equation}
 for functions $\phi_{ij}\colon U_{ij}\to \bC^*$. The integrality assumption implies that, after replacing $\phi_{ij}$ by~${r_{ij}\cdot \phi_{ij}}$ for constants $r_{ij}\in \bC^*$, we can assume that
 \[
\phi_{ij}|_{U_{ijk}}\cdot \phi_{jk}|_{U_{ijk}}\cdot \phi_{ki}|_{U_{ijk}}=1.\]
We then define the line bundle $L$ by gluing the trivial line bundles $L_i$ over $U_i$ using multiplication by $\phi_{ij}$. The relations \eqref{dull} show that the connections $\nabla_i={\rm d}+\Theta_i$ on $L_i$ glue to a connection~$\nabla$ on $L$. Since $\nabla_i$ has curvature ${\rm d}\Theta_i=\Omega|_{U_i}$, the glued connection $\nabla$ has curvature $\Omega$.
\end{proof}

\begin{Remarks}\quad
\begin{itemize}\itemsep=0pt
\item[(i)]
The relation \eqref{dull} can be phrased as the statement that the gluing map $\phi_{ij}$ for the line bundle $L$ is the exponential generating function relating the symplectic potentials $\Theta_i|_{U_{ij}}$ and $\Theta_j|_{U_{ij}}$ on $U_{ij}$.
\item[(ii)] Suppose $U\subset M$ is a contractible open subset. Then sections $s\in H^0(U,L)$ up to scale are in bijection with symplectic potentials $\Theta$ on $U$. Given $s$ we can write ${\nabla(s)=\Theta\cdot s}$ with~${{\rm d}\Theta=\Omega}$. Conversely, given another symplectic potential $\Theta'$ on $U$ we can write~$\Theta'-\Theta={\rm d}\log(f)$ and hence define a section $s'=f\cdot s$ satisfying $\nabla(s')=\Theta'\cdot s'$.
\end{itemize}
\end{Remarks}

\subsection*{Acknowledgements}
The ideas presented here have evolved from discussions with many people over a long period of time. I would particularly like to thank Sergei Alexandrov, Andy Neitzke, Boris Pioline and J{\"o}rg Teschner for sharing their insights, and for patiently explaining many basic things to me. I~am also very grateful for discussions and correspondence with Murad Alim, Fabrizio Del Monte, Maciej Dunajski, Lotte Hollands, Kohei Iwaki, Omar Kidwai, Dima Korotkin, Oleg Lisovyy, Lionel Mason, Ian Strachan and Menelaos Zikidis. Finally, I~thank the anonymous referees for their careful reading and useful suggestions for improvements.

\pdfbookmark[1]{References}{ref}
\LastPageEnding

\end{document}